\documentclass[a4paper,12pt]{article}
\usepackage[cp1251]{inputenc}
\usepackage{amsfonts, amssymb, amsmath, amsthm, amscd}
\usepackage{graphicx}
\usepackage{cite}
\ifx\undefined \pdfgentounicode \else
\input{glyphtounicode} \pdfgentounicode=1
\fi

\author{A.A. Vasil'eva\footnote{Lomonosov Moscow State University, Moscow Center for Fundamental and Applied Mathematics}}
\title{Linear widths of weighted Sobolev classes with conditions on the highest order and zero derivatives\footnote{The research was carried out with the
financial support of the Russian Foundation for Basic Research
(grant no. 19-01-00332).}}
\date{}
\begin{document}

\maketitle

\newenvironment{Biblio}{%
                  \renewcommand{\refname}{\footnotesize REFERENCES}%
                  \begin{thebibliography}{99}%
                  \itemsep -0.2em}%
                  {\end{thebibliography}}

\def\inff{\mathop{\smash\inf\vphantom\sup}}
\renewcommand{\le}{\leqslant}
\renewcommand{\ge}{\geqslant}
\newcommand{\sgn}{\mathrm {sgn}\,}
\newcommand{\inter}{\mathrm {int}\,}
\newcommand{\dist}{\mathrm {dist}}
\newcommand{\supp}{\mathrm {supp}\,}
\newcommand{\R}{\mathbb{R}}
\newcommand{\Z}{\mathbb{Z}}
\newcommand{\N}{\mathbb{N}}
\newcommand{\Q}{\mathbb{Q}}
\theoremstyle{plain}
\newtheorem{Trm}{Theorem}
\newtheorem{trma}{Theorem}
\newtheorem{Def}{Definition}
\newtheorem{Cor}{Corollary}
\newtheorem{Lem}{Lemma}
\newtheorem{Rem}{Remark}
\newtheorem{Sta}{Proposition}
\newtheorem{Sup}{Assumption}
\newtheorem{Supp}{Assumption}
\newtheorem{Exa}{Example}
\renewcommand{\proofname}{\bf Proof}
\renewcommand{\thetrma}{\Alph{trma}}
\renewcommand{\theSupp}{\Alph{Supp}}

\begin{abstract}
In this paper order estimates for the linear widths of some
function classes are obtained; these classes are defined by
restrictions on the weighted $L_{p_1}$-norm of the $r$th
derivative and the weighted $L_{p_0}$-norm of zero derivative.
\end{abstract}

Keywords: weighted Sobolev classes, function class intersections,
linear widths

\section{Introduction}

In \cite{vas_inters} order estimates for the Kolmogorov widths of
some weighted Sobolev classes with conditions on the highest order
and zero derivatives in a weighted Lebesgue space were obtained.
These classes are defined as
\begin{align}
\label{m_def} M = \left\{ f:\Omega \rightarrow \R:\; \left\|
\frac{\nabla ^r f}{g}\right\|_{L_{p_1}(\Omega)}\le 1, \quad
\|wf\|_{L_{p_0}(\Omega)}\le 1\right\}
\end{align}
(here $\Omega\subset \R^d$ is a domain, $r\in \N$, $1<p_0, \,
p_1\le \infty$, $g$, $w:\Omega \rightarrow (0, \, \infty)$ are
measurable functions). The weighted Lebesgue space is defined as
$$
L_{q,v}(\Omega)=\left\{f:\Omega \rightarrow \R| \; \ \| f\|
_{L_{q,v}(\Omega)}<\infty\right\}, \quad \text{where}\quad \|
f\|_{L_{q,v}(\Omega)}{=}\| fv\|_{L_q(\Omega)};
$$
here $1\le q<\infty$, $v:\Omega \rightarrow (0, \, \infty)$ is a
measurable function. Three examples were considered. In the first
two of them, $\Omega$ is a bounded John domain, the weights are
functions of the distance from an $h$-set $\Gamma \subset
\partial \Omega$. In the third example, $\Omega = \R^d$, the
weights have the form
\begin{align}
\label{gw_2} g(x)=(1+|x|)^\beta, \quad w(x) = (1+|x|)^\sigma,
\quad v(x) = (1 + |x|)^\lambda.
\end{align}
In this paper we obtain the estimates for the linear widths of a
set $M$ in the space $L_{q,v}(\Omega)$.

The problem on estimating the Kolmogorov and linear widths of
weighted Sobolev classes with different constraints on the
derivatives was studied in \cite{triebel12, mieth1, mieth2,
vas_width_raspr, lo1, lo2, lo3, triebel, trieb_mat_sb, boy_1,
myn_otel, ait_kus1, ait_kus2}. For details, see \cite{vas_inters}.

We give the necessary definitions.

Let $X$ be a normed space. By $L(X, \, X)$ we denote the family of
linear continuous operators on $X$, by ${\rm rk}\, A$, the
dimension of the range of an operator $A$. Let $C\subset X$, $n\in
\Z_+$. The linear $n$-width of a set $C$ in the space $X$ is
defined as
$$
\lambda_n(C, \, X) = \inf _{A\in L(X, \, X), \, {\rm rk}\, A\le n}
\sup _{x\in C}\|x-Ax\|.
$$

For $1\le p\le \infty$, we write $p'=\frac{p}{p-1}$.

\begin{Def}
\label{theta_j} Given $s_*$, $\tilde \theta$, $\hat \theta\in \R$,
we define the numbers $j_0\in \N$ and $\theta_j\in \R$ $(1\le j\le
j_0)$ as follows.
\begin{enumerate}
\item Let $p_0\ge q$, $p_1\ge q$. Then $j_0=2$, $\theta_1=s_*$, $\theta_2 = \tilde
\theta$.
\item Let $p_0>q$, and, in addition, $p_1<q\le 2$ or $2\le p_1<q$. Then $j_0=3$,
$\theta_1=s_*+\frac 1q-\frac{1}{p_1}$, $\theta_2 = \tilde \theta$,
$\theta_3 =\hat \theta$.
\item Let $p_0\le q$, $p_1\le q$, and, in addition, $q\le 2$ or $\min \{p_0, \, p_1\}\ge 2$.
Then $j_0=2$, $\theta_1= s_*+\frac 1q-\frac {1}{p_1}$, $\theta_2 =
\hat \theta$.
\item Let $p_1>q$, and, in addition, $p_0<q\le 2$ or $2\le p_0<q$. Then $j_0=3$, $\theta_1 = s_*$,
$\theta_2=\tilde \theta$, $\theta_3 =\hat \theta$.
\item Let $p_0\le 2<q$, $p_1\le 2<q$, $\frac{1}{p_0}+\frac 1q\ge
1$, $\frac{1}{p_1}+\frac 1q\ge 1$. Then $j_0=4$, $\theta_1 =s_* +
\frac 12 -\frac {1}{p_1}$, $\theta_2 =\frac{q(s_*+1/q-1/p_1)}{2}$,
$\theta_3 =\hat \theta +\frac 12 -\frac 1q$, $\theta_4
=\frac{q\hat \theta}{2}$.
\item Let $p_0< p_1\le 2<q$, $\frac{1}{p_0}+\frac 1q\le
1$, $\frac{1}{p_1}+\frac 1q\le 1$.
\begin{enumerate}
\item If $p_0<p_1<2$, then $j_0=5$, $\theta_1 = s_*+\frac 1q-\frac 12$,
$\theta_2=\frac{p_1'(s_*+1/q-1/p_1)}{2}$, $\theta_3 = \hat \theta
+\left(\frac{1}{p_0}-\frac{1}{p_1}\right)\left(1-\frac{\tilde
\theta}{s_*}\right)+\frac{1}{p_1}-\frac 12$, $\theta_4=
\frac{p_1'}{2}\left(\hat \theta
+\left(\frac{1}{p_0}-\frac{1}{p_1}\right)\left(1-\frac{\tilde
\theta}{s_*}\right)\right)$, $\theta_5 =
\frac{p_0'\hat\theta}{2}$.
\item If $p_0<p_1=2$, then $j_0=3$, $\theta_1 = s_*+\frac 1q-\frac
12$, $\theta_2 = \hat \theta
+\left(\frac{1}{p_0}-\frac{1}{2}\right)\left(1-\frac{\tilde
\theta}{s_*}\right)$, $\theta_3 = \frac{p_0'\hat\theta}{2}$.
\end{enumerate}
\item Let $p_1< p_0\le 2<q$, $\frac{1}{p_0}+\frac 1q\le
1$, $\frac{1}{p_1}+\frac 1q\le 1$.
\begin{enumerate}
\item If $p_1<p_0<2$, then $j_0=5$, $\theta_1 = s_*+\frac 1q-\frac 12$,
$\theta_2=\frac{p_1'(s_*+1/q-1/p_1)}{2}$, $\theta_3 = \hat \theta
+\left(\frac{1}{p_1}-\frac{1}{p_0}\right)\frac{\tilde
\theta}{s_*}+\frac{1}{p_0}-\frac12$, $\theta_4=
\frac{p_0'}{2}\left(\hat \theta
+\left(\frac{1}{p_1}-\frac{1}{p_0}\right)\frac{\tilde
\theta}{s_*}\right)$, $\theta_5 = \frac{p_1'\hat\theta}{2}$.
\item If $p_1<p_0=2$, then $j_0=4$, $\theta_1 = s_*+\frac 1q-\frac 12$,
$\theta_2=\frac{p_1'(s_*+1/q-1/p_1)}{2}$, $\theta_3 = \hat \theta
+\left(\frac{1}{p_1}-\frac{1}{2}\right)\frac{\tilde \theta}{s_*}$,
$\theta_4 = \frac{p_1'\hat\theta}{2}$.
\end{enumerate}
\item Let $p:=p_1= p_0< 2<q$, $\frac{1}{p}+\frac 1q\le 1$. Then $j_0=4$, $\theta_1 =
s_*+\frac 1q-\frac 12$, $\theta_2=\frac{p'(s_*+1/q-1/p)}{2}$,
$\theta_3 = \hat \theta +\frac{1}{p}-\frac 12$, $\theta_4 =
\frac{p'\hat\theta}{2}$.
\item Let $p_0< 2<q$, $p_1\le 2<q$, $\frac{1}{p_0}+\frac
1q> 1$, $\frac{1}{p_1}+\frac 1q< 1$.
\begin{enumerate}
\item If $p_1<2$, then $j_0=6$, $\theta_1=s_*+\frac 1q-\frac 12$, $\theta_2=
\frac{p_1'(s_*+1/q-1/p_1)}{2}$, $\theta_3=\hat \theta
+\left(\frac{1}{p_0}-\frac{1}{p_1}\right)\left(1-\frac{\tilde
\theta}{s_*}\right)+\frac{1}{p_1}-\frac 12$, $\theta_4=
\frac{p_1'}{2}\left(\hat \theta
+\left(\frac{1}{p_0}-\frac{1}{p_1}\right)\left(1-\frac{\tilde
\theta}{s_*}\right)\right)$, $\theta_5= \hat \theta +\frac
12-\frac 1q$, $\theta_6=\frac{q\hat \theta}{2}$.
\item If $p_1=2$, then $j_0=4$, $\theta_1=s_*+\frac 1q-\frac 12$, $\theta_2=\hat \theta
+\left(\frac{1}{p_0}-\frac{1}{2}\right)\left(1-\frac{\tilde
\theta}{s_*}\right)$, $\theta_3= \hat \theta +\frac 12-\frac 1q$,
$\theta_4=\frac{q\hat \theta}{2}$.
\end{enumerate}
\item Let $p_0\le 2<q$, $p_1< 2<q$, $\frac{1}{p_0}+\frac
1q< 1$, $\frac{1}{p_1}+\frac 1q> 1$.
\begin{enumerate}
\item If $p_0<2$, then $j_0=6$, $\theta_1 =s_* + \frac 12 -\frac {1}{p_1}$,
$\theta_2 =\frac{q(s_*+1/q-1/p_1)}{2}$, $\theta_3 = \hat \theta
+\left(\frac{1}{p_1}-\frac{1}{p_0}\right)\frac{\tilde
\theta}{s_*}+\frac{1}{p_0}-\frac12$, $\theta_4=
\frac{p_0'}{2}\left(\hat \theta
+\left(\frac{1}{p_1}-\frac{1}{p_0}\right)\frac{\tilde
\theta}{s_*}\right)$, $\theta_5=\hat \theta +\frac 12-\frac 1q$,
$\theta_6=\frac{q\hat \theta}{2}$.
\item If $p_0=2$, then $j_0=5$, $\theta_1 =s_* + \frac 12 -\frac {1}{p_1}$,
$\theta_2 =\frac{q(s_*+1/q-1/p_1)}{2}$, $\theta_3 = \hat \theta
+\left(\frac{1}{p_1}-\frac{1}{2}\right)\frac{\tilde \theta}{s_*}$,
$\theta_4=\hat \theta +\frac 12-\frac 1q$, $\theta_5=\frac{q\hat
\theta}{2}$.
\end{enumerate}
\item Let $2<p_0<q$, $p_1<2$, $\frac{1}{p_1}+\frac 1q\le 1$. Then
$j_0=4$, $\theta_1=s_*+\frac 1q-\frac 12$, $\theta_2 =
\frac{p_1'(s_*+1/q-1/p_1)}{2}$, $\theta_3 = \hat \theta
+\left(\frac{1}{p_1}-\frac 12\right)\frac{\tilde \theta}{s_*}$,
$\theta_4 = \frac{p_1'\hat \theta}{2}$.
\item Let $2<p_1<q$, $p_0<2$, $\frac{1}{p_0}+\frac 1q\le 1$. Then
$j_0=3$, $\theta_1 = s_*+\frac 1q-\frac{1}{p_1}$, $\theta_2 = \hat
\theta+\left(\frac{1}{p_0}-\frac 12\right)\left(1-\frac{\tilde
\theta}{s_*}\right)$, $\theta_3= \frac{p_0'\hat \theta}{2}$.
\end{enumerate}
\end{Def}

Notice that here for $q>2$ the following cases are not considered:
a) $p_0>q$, $p_1\le 2$, b) $p_1>q$, $p_0\le 2$, c) $p_0>2$,
$\frac{1}{p_1}+\frac{1}{q}>1$, d) $p_1>2$,
$\frac{1}{p_0}+\frac{1}{q}>1$.

We formulate the theorem about order estimates of the linear
widths for the third example from \cite{vas_inters}.

\begin{Trm}
\label{main3} Let $\Omega=\R^d$, $r\in \N$, $1<p_0, \, p_1\le
\infty$, $1<q<\infty$; suppose that $p_0$, $p_1$, $q$ satisfy one
of the conditions from Definition \ref{theta_j}. Let (\ref{gw_2})
hold, $\frac rd+\min\left\{\frac 1q, \,
\frac{1}{p_0}\right\}-\frac{1}{p_1}>0$, $\min\{\beta+\sigma
+d/p_0-d/p_1, \, \beta + \sigma\}>0$. We set $s_*=\frac rd$,
$$
\tilde \theta = \frac rd \cdot \frac{\sigma -\lambda +
\frac{d}{p_0}-\frac{d}{q}}{\beta +\sigma +r+\frac{d}{p_0} -
\frac{d}{p_1}},
$$
$$
\hat \theta = \frac{\sigma\left(\frac{r}{d}+\frac 1q-
\frac{1}{p_1}\right) +\beta\left(\frac 1q- \frac{1}{p_0}\right)
-\lambda \left(\frac rd +\frac{1}{p_0}
-\frac{1}{p_1}\right)}{\beta +\sigma +r+\frac{d}{p_0} -
\frac{d}{p_1}},
$$
$\mathfrak{Z}=(r, \, d, \, p_0, \, p_1, \, q, \, \beta, \, \sigma,
\, \lambda)$. Suppose that $\tilde \theta>0$ for $p_0\ge q$,
$\hat\theta>0$ for $p_0<q$. Let the set $M$ be defined by
(\ref{m_def}), and let $j_0\in \N$ and $\theta_j$ be defined
according to Definition \ref{theta_j}. Suppose that there is
$j_*\in \{1, \, \dots, \, j_0\}$ such that $\theta_{j_*}< \min
_{j\ne j_*} \theta_j$. Then
$$
\lambda_n(M, \, L_{q,v}(\R^d)) \underset{\mathfrak{Z}}{\asymp}
n^{-\theta_{j_*}}.
$$
\end{Trm}

This theorem generalizes and refines the result from
\cite{myn_otel} (p. 165, Theorem 9, except the case $n=\delta_2$).

For the first and the second example from \cite{vas_inters}, the
theorem can be formulated similarly; the number $\tilde\theta$ is
defined, respectively, by \cite[formulas (5) and
(11)]{vas_inters}; the number $\hat \theta$ is defined,
respectively, by \cite[formulas (6) and (12)]{vas_inters}.

The paper is organized as follows. In \S 2 upper estimates for the
linear widths of the function classes $BX_{p_1}(\Omega)\cap
BX_{p_0}(\Omega)$ are obtained (these classes were introduced in
\cite{vas_inters}). In \S 3 the lower estimates are proved. This
together with arguments from \cite[\S 4, 5]{vas_inters} yields
Theorem \ref{main3}, as well as order estimates for the linear
widths for the other two examples from \cite{vas_inters}.

\section{Upper estimates for the linear widths of the classes $BX_{p_1}(\Omega)\cap BX_{p_0}(\Omega)$.}

The spaces $X_{p_i}(\Omega)$ ($i=0, \, 1$) and $Y_q(\Omega)$, the
subspace ${\cal P}(\Omega)$ of dimension $r_0\in \N$, the numbers
$c\ge 1$, $s_*>\left(\frac{1}{p_1} -\frac 1q\right)_+$,
$\gamma_*>0$, $k_*\in \N$, $\alpha_*\in \R$, $\mu_*\in \R$ are as
in \cite[\S 2]{vas_inters}; we suppose that Assumptions A--F from
\cite{vas_inters} hold; in addition, we suppose that the
projections $P_E$ are continuous.

We set
\begin{align}
\label{til_theta} \tilde \theta = \frac{s_*\left(\alpha_*
+\frac{\gamma_*}{p_0} -\frac{\gamma_*}{q}\right)}{\mu_* + \alpha_*
+\gamma_*\left(s_* +\frac{1}{p_0} -\frac{1}{p_1}\right)}, \quad
\hat \theta = \frac{\alpha_*\left(s_* +\frac 1q
-\frac{1}{p_1}\right) +\mu_*\left(\frac 1q
-\frac{1}{p_0}\right)}{\mu_* + \alpha_* +\gamma_*\left(s_*
+\frac{1}{p_0} -\frac{1}{p_1}\right)},
\end{align}
$$
\mathfrak{Z}_0=(p_0, \, p_1, \, q, \, r_0, \, c, \, k_*,\, s_*, \,
\gamma_*, \, \mu_*, \, \alpha_*).
$$
The numbers $j_0\in \N$ and $\theta_j\in \R$ ($1\le j\le j_0$) are
defined according to Definition \ref{theta_j}.

\begin{Trm}
\label{trm} Let
\begin{align}
\label{s1qp} \min \left\{s_*, \, s_* +\frac 1q -\frac{1}{p_1},\;
s_*+\frac{1}{p_0} -\frac{1}{p_1}\right\}>0,
\end{align}
\begin{align}
\label{mua} \min\left\{\mu_* +\alpha_*+\frac{\gamma_*}{p_0}
-\frac{\gamma_*}{p_1}, \, \mu_*+\alpha_* \right\}>0.
\end{align}
Suppose that $\alpha_*> \frac{\gamma_*}{q} -\frac{\gamma_*}{p_0}$
for $p_0 \ge q$, $\alpha_*\left(s_*+\frac 1q-\frac{1}{p_1}\right)>
\mu_*\left(\frac{1}{p_0}-\frac 1q\right)$ for $p_0\le q$, $p_1\le
q$ or for $p_0<q$, $p_1>q$. Let $p_0$, $p_1$, $q$ satisfy one of
the conditions 1--12 from Definition \ref{theta_j}. Suppose that
there exists $j_*\in \{1, \, \dots, \, j_0\}$ such that
$\theta_{j_*}< \min _{j\ne j_*} \theta_j$. Then
$$
\lambda_n(BX_{p_0}(\Omega) \cap BX_{p_1}(\Omega), \, Y_q(\Omega))
\underset{\mathfrak{Z}_0}{\lesssim} n^{-\theta_{j_*}}.
$$
\end{Trm}

\begin{proof}
In \cite[p. 8]{vas_inters} the partition $\{\Omega_t\}_{t\ge t_0}$
of $\Omega$, the operators $P_{t,m}:Y_q(\Omega)\rightarrow
Y_q(\Omega)$ and the numbers $\nu_{t,m}$, $\nu'_{t,m}\in \Z_+$
were defined. Here we notice their following properties:
\begin{align}
\label{supp} (P_{t,m}f)|_{\Omega\backslash \Omega_t}=0, \quad f\in
Y_q(\Omega), \quad t\ge t_0,
\end{align}
\begin{align}
\label{rk_ptm} {\rm rk}\, P_{t,m}\le \nu'_{t,m}
\underset{\mathfrak{Z}_0}{\lesssim} 2^{\gamma _*k_*t}\cdot 2^m,
\quad {\rm rk}\, (P_{t,m+1}-P_{t,m})\le \nu_{t,m}
\underset{\mathfrak{Z}_0}{\lesssim} 2^{\gamma _*k_*t}\cdot 2^m.
\end{align}

Taking into account that the operators $P_{t,m}$ are linear and
continuous (since the projections $P_E$ are continuous), and
repeating the arguments from \cite[pp. 14--16,
19--21]{vas_inters}, we obtain:
\begin{enumerate}
\item if $p_0\ge q$, $p_1\ge q$, the linear widths are estimated
as the Kolmogorov widths; this implies the upper estimate in case
1 of Definition \ref{theta_j};

\item if $p_1<q<p_0$ and, in addition, $q\le 2$ or $p_1\ge 2$, the linear widths are estimated
as the Kolmogorov widths for $p_1<q<p_0$, $q\le 2$; this implies
the upper estimate in case 2 of Definition \ref{theta_j};

\item if $p_0\le q$, $p_1\le q$ and, in addition, $q\le 2$ or $\min\{p_0, \, p_1\}\ge
2$, the linear widths are estimated as the Kolmogorov widths for
$p_0\le q\le 2$, $p_1\le q\le 2$; this implies the upper estimate
in case 3 of Definition \ref{theta_j};

\item if $p_0<q<p_1$ and, in addition, $q\le 2$ or $p_0\ge 2$, the linear widths are estimated as the Kolmogorov
widths for $p_0<q<p_1$, $q\le 2$; this implies the upper estimate
in case 4 of Definition \ref{theta_j}.
\end{enumerate}

Hence, it remains to consider the case $q>2$, $\min\{p_1, \,
p_0\}<2$.

Let us formulate the corollary from E.D. Gluskin's theorem about
the linear widths of the finite-dimensional balls
\cite{bib_gluskin}.
\begin{trma}
\label{gl_teor_lin} {\rm \cite{bib_gluskin}.} If $1\le p<2<q\le
\infty$, $(p, \, q)\ne (1, \, \infty)$, $n\le N/2$, then
\begin{align}
\label{glusk_ln} \lambda_n(B_p^N, \, l_q^N) \underset{p,q}{\asymp}
\min \{1, \, n^{-1/2}N^{\max\{1/q, 1/p'\}}\}.
\end{align}
If $1\le p\le q\le 2$ or $2\le p\le q \le\infty$, $n\le N/2$, then
\begin{align}
\label{glusk_ln1} \lambda_n(B_p^N, \, l_q^N)
\underset{p,q}{\asymp} 1.
\end{align}
The upper estimates also hold for $N/2< n\le N$.
\end{trma}

For $1\le q\le p\le \infty$ the following equation holds
\cite{pietsch1}, \cite{stesin}:
\begin{align}
\label{pietsch_stesin} \lambda_n(B_p^N, \, l_q^N) = (N-n)^{\frac
1q-\frac 1p}.
\end{align}

Let $W_{t,m}$ be the set of sequences $(c_j)_{j=1} ^{\nu_{t,m}}$
such that
\begin{align}
\label{w_tm} \begin{array}{c} \left(\sum \limits
_{j=1}^{\nu_{t,m}}|c_j|^{p_1} \right)^{1/p_1} \le
2^{\mu_*k_*t}\cdot 2^{-m(s_*+1/q-1/p_1)}, \\ \left(\sum \limits
_{j=1}^{\nu_{t,m}}|c_j|^{p_0} \right)^{1/p_0} \le
2^{-\alpha_*k_*t}\cdot 2^{m(1/p_0-1/q)}.\end{array}
\end{align}
The following assertion can be proved similarly to Proposition 2
from \cite{vas_inters}.

\begin{Sta}
Let $l\in \Z_+$. Then
\begin{align}
\label{dl} \lambda_l((P_{t,m+1}-P_{t,m})(BX_{p_1}(\Omega)\cap
BX_{p_0}(\Omega)), \, Y_q(\Omega))
\underset{\mathfrak{Z}_0}{\lesssim} \lambda_l(W_{t,m}, \,
l_q^{\nu_{t,m}}),
\end{align}
\begin{align}
\label{dl1} \lambda_l(P_{t,m}(BX_{p_1}(\Omega)\cap
BX_{p_0}(\Omega)), \, Y_q(\Omega))
\underset{\mathfrak{Z}_0}{\lesssim} 2^{-\alpha_*k_*t}\cdot
2^{-m(1/q-1/p_0)}\lambda_l(B_{p_0}^{\nu'_{t,m}}, \,
l_q^{\nu'_{t,m}}).
\end{align}
\end{Sta}

We also need the following assertion: if $\frac{1}{q'} =
\frac{1-\lambda}{p_1} + \frac{\lambda}{p_0}$, $\lambda \in (0, \,
1)$, then
\begin{align}
\label{wkm_lq} W_{k,m} \subset
2^{k_*t\left((1-\lambda)\mu_*-\lambda \alpha_*\right)} \cdot
2^{-m((1-\lambda)(s_*+1/q-1/p_1)+\lambda(1/q-1/p_0))}B_{q'}^{\nu_{t,m}}.
\end{align}
It is a particular case of Galeev's result \cite[Theorem
2]{galeev1} or the corollary of H\"{o}lder's inequality.

We define the numbers $\hat m_t$, $\tilde m_t$, $m_t\in \R$ by the
equations
\begin{align}
\label{hat_mt} 2^{\gamma_*k_*t}\cdot 2^{\hat m_t}=n,
\end{align}
\begin{align}
\label{til_mt_t} 2^{\tilde m_t s_*} = 2^{(\mu_* + \alpha_*
+\gamma_*/p_0 -\gamma_*/p_1)k_*t},
\end{align}
\begin{align}
\label{mt_t} 2^{m_t(s_*+1/p_0-1/p_1)} = 2^{(\mu_*+\alpha_*) k_*t}.
\end{align}
Further, we define the numbers $m_t^{(q)}$, $m_t^{(p_0')}$,
$m_t^{(p_1')}$ by the equations
\begin{align}
\label{line_mt} 2^{\gamma_*k_*t}\cdot 2^{m_t^{(q)}}=n^{q/2}, \quad
2^{\gamma_*k_*t}\cdot 2^{m_t^{(p_0')}}=n^{p_0'/2}, \quad
2^{\gamma_*k_*t}\cdot 2^{m_t^{(p_1')}}=n^{p_1'/2},
\end{align}
and set $\overline{m}_t=\max \{m_t^{(q)}, \, m_t^{(p_0')}, \,
m_t^{(p_1')}\}$.

The numbers $t_{***}(n)$, $\varepsilon>0$, $t_1(n)\in [0, \,
t_{***}(n)]$ and $m_1(n)$ will be chosen later in dependence on
$\mathfrak{Z}_0$. Here $2^{t_{***}(n)}$ will be a positive power
of $n$, depending on $\mathfrak{Z}_0$, $2^{m_1(n)}\in [1, \,
n^{\max\{q, \, p_0', \, p_1'\}/2}]$.

We set
$$
2^{m_t^*} = \max \{2^{\hat m_t}\cdot 2^{-\varepsilon |t-t_1(n)|},
\, 1\},
$$
$$
l_{t,m} = \left\{ \begin{array}{l} \lceil n\cdot
2^{-\varepsilon(|m-m_1(n)|+|t-t_1(n)|)}\rceil, \quad 0\le t\le
t_{***}(n), \; m_t^*\le m\le \overline{m}_t, \\ 0, \quad
\text{otherwise}.
\end{array}\right.
$$
For a sufficiently small $\varepsilon>0$ we have
\begin{align}
\label{sum_ltm} \sum \limits _{t\ge 0, \, m\ge m_t^*}
l_{t,m}\underset{\varepsilon, \, \mathfrak{Z}_0}{\lesssim} n.
\end{align}

We set $\tilde \Omega_t =\cup_{l\ge t} \Omega_l$. From
(\ref{supp}) it follows that for $f\in Y_q(\Omega)$
\begin{align}
\label{f_razl} f = \sum \limits _{t_0\le t\le [t_{***}(n)]}
P_{t,[m_t^*]}f + \sum \limits _{t_0\le t\le [t_{***}(n)]} \sum
\limits _{m\ge [m_t^*]} (P_{t,m+1}f - P_{t,m}f) + f\cdot \chi
_{\tilde \Omega _{[t_{***}(n)]+1}}.
\end{align}

By (\ref{rk_ptm}), (\ref{dl}), (\ref{dl1}), (\ref{hat_mt}),
(\ref{sum_ltm}), (\ref{f_razl}), in order to proof Theorem
\ref{trm} it suffices to estimate from above
\begin{align}
\label{sum} \begin{array}{c} S:=\sum \limits _{t_0\le t\le
t_{***}(n)} \sum \limits _{m\ge [m_t^*]}
\lambda_{l_{t,m}}(W_{t,m}, \, l_q^{\nu_{t,m}}) +\\+ \sum \limits
_{t\le t_{***}(n):\; m_t^*=0} 2^{-\alpha_*k_*t}
\lambda_{l_{t,0}}(B_{p_0}^{\nu'_{t,0}}, \, l_q^{\nu'_{t,0}})+
\|f\|_{Y_q(\tilde \Omega _{[t_{***}(n)]+1})}.
\end{array}
\end{align}

In \cite[Proposition 4]{vas_inters} it was shown that, for $p_1\le
q$, $p_0\le q$ and $f\in X_{p_0}(\Omega)\cap X_{p_1}(\Omega)$,
\begin{align}
\label{p0lq_emb} \|f\|_{Y_q(\tilde\Omega_t)}
\underset{\mathfrak{Z}_0}{\lesssim}
2^{-\frac{\alpha_*(s_*+1/q-1/p_1)-\mu_*(1/p_0-1/q)}
{s_*+1/p_0-1/p_1}k_*t}.
\end{align}

{\bf The case $q>2$, $\max\{p_0, \, p_1\}<2$,
$\frac{1}{p_0}+\frac{1}{q}\ge 1$, $\frac{1}{p_1} +\frac{1}{q}\ge
1$.} It is well-known that $\lambda_n(B_p^N, \, l_q^N)
\underset{p,q}{\asymp} d_n(B_p^N, \, l_q^N)$ for $\frac 1p+\frac
1q \ge 1$ (see \cite{bib_gluskin}). Hence, the estimates for the
linear widths of $\lambda_n(BX_{p_0}(\Omega) \cap
BX_{p_1}(\Omega), \, Y_q(\Omega))$ can be obtained as for the
Kolmogorov widths in \cite[p. 21]{vas_inters}. This implies the
upper estimate for case 5 of Definition \ref{theta_j}.

{\bf The case $q>2$, $\max\{p_0, \, p_1\}<2$,
$\frac{1}{p_0}+\frac{1}{q}\le 1$, $\frac{1}{p_1} +\frac{1}{q}\le
1$, $p_1> p_0$.} We define the numbers $m_t'$ by the equation
\begin{align}
\label{mtpr} 2^{\mu_*k_*t}\cdot 2^{-m'_t(s_*+1/q-1/p_1)} =
2^{-\alpha_* k_*t}\cdot 2^{-m_t'(1/q-1/p_0)}\cdot
n^{-1/2}\cdot2^{\gamma_*k_*t/p_0'}\cdot 2^{m_t'/p_0'};
\end{align}
the numbers $t_*(n)$, $t_{**}(n)$, $t_{***}(n)$ are defined by the
equations
\begin{align}
\label{t_n_st} \tilde m_{t_*(n)} = \hat m_{t_*(n)}, \quad \tilde
m_{t_{**}(n)} = m_{t_{**}(n)}^{(p_1')}, \quad m_{t_{***}(n)} =
m_{t_{***}(n)}^{(p_0')}.
\end{align}
Then
\begin{align}
\label{m_t_pr} m'_{t_{**}(n)} \stackrel{(\ref{til_mt_t}),
(\ref{line_mt})}{=} \tilde m_{t_{**}(n)}, \quad m'_{t_{***}(n)}
\stackrel{(\ref{mt_t}), (\ref{line_mt})}{=} m_{t_{***}(n)},
\end{align}
\begin{align}
\label{tst}
2^{(\mu_*+\alpha_*+\gamma_*(s_*+1/p_0-1/p_1))k_*t_*(n)}
\stackrel{(\ref{hat_mt}),(\ref{til_mt_t})}{=} n^{s_*},
\end{align}
\begin{align}
\label{tsst}
2^{(\mu_*+\alpha_*+\gamma_*(s_*+1/p_0-1/p_1))k_*t_{**}(n)}
\stackrel{(\ref{til_mt_t}), (\ref{line_mt})}{=} n^{s_*p_1'/2},
\end{align}
\begin{align}
\label{tssst}
2^{(\mu_*+\alpha_*+\gamma_*(s_*+1/p_0-1/p_1))k_*t_{***}(n)}
\stackrel{(\ref{mt_t}), (\ref{line_mt})}{=}
n^{(s_*+1/p_0-1/p_1)p_0'/2}.
\end{align}

Now we estimate the sum (\ref{sum}).

Let
$$
{\rm I} = \{(t, \, m):\;t\ge 0, \; m_t^*\le m\le m_t^{(p_1')}, \;
m\ge \tilde m_t\},
$$
$$
{\rm II} = \{(t, \, m):\; 0\le t\le t_{***}(n), \; m\ge
m_t^{(p_1')}, \; m\ge m_t'\},
$$
$$
{\rm III} = \{(t, \, m):\; 0\le t\le t_{***}(n), \; m_t^*\le m\le
\tilde m_t, \; m\le m_t'\}.
$$
In ${\rm I}$ and ${\rm II}$ we use the inclusion $W_{t,m}\subset
2^{\mu_*k_*t}\cdot 2^{-m(s_*+1/q-1/p_1)}B^{\nu_{t,m}}_{p_1}$, in
${\rm III}$ we apply the inclusion $W_{t,m} \subset
2^{-\alpha_*k_*t}\cdot 2^{-m(1/q-1/p_0)}B^{\nu_{t,m}}_{p_0}$.

From Theorem \ref{gl_teor_lin}, (\ref{til_theta}),
(\ref{p0lq_emb}) and (\ref{tssst}), we get
$$
S \underset{\mathfrak{Z}_1}{\lesssim} n^{-\hat \theta p_0'/2}+\sum
\limits _{(t,m)\in {\rm I}} 2^{\mu_*k_*t}\cdot
2^{-m(s_*+1/q-1/p_1)} l_{t,m}^{-1/2} \cdot
2^{\gamma_*k_*t/p_1'}\cdot 2^{m/p_1'}+
$$
$$
+\sum \limits _{(t, \, m)\in {\rm II}} 2^{\mu_*k_*t}\cdot
2^{-m(s_*+1/q-1/p_1)} + \sum \limits _{(t, \, m)\in {\rm III}}
2^{-\alpha_*k_*t}\cdot 2^{-m(1/q-1/p_0)} l_{t,m}^{-1/2} \cdot
2^{\gamma_*k_*t/p_0'}\cdot 2^{m/p_0'} =:S'.
$$
In the second sum, there is a decreasing geometric progression in
$m$, and in the third sum, there is an increasing geometric
progression in $m$. Also notice that
\begin{align}
\label{eq} \begin{array}{c} s_*\frac{\mu_*+\gamma_*\left(s_*+\frac
1q -\frac{1}{p_1}\right)} {\mu_*+\alpha_*+\gamma_*\left(s_*+ \frac
{1}{p_0} -\frac{1}{p_1}\right)} -s_* - \frac 1q +\frac{1}{p_1} =
\\ = -\hat \theta -\left(\frac{1}{p_0}-\frac{1}{p_1}\right)
\left(1-\frac{\tilde \theta}{s_*}\right).
\end{array}
\end{align}

For $\varepsilon=0$ we calculate the values of summands at the
points $(t, \, m)=(0, \, \hat m_0)$, $(0, \, m_0^{(p_1')})$,
$(t_*(n), \, \hat m_{t_*(n)})$, $(t_{**}(n), \,
m^{(p_1')}_{t_{**}(n)})$, $(t_{***}(n), \,
m^{(p_0')}_{t_{***}(n)})$ (here we take into account
(\ref{til_mt_t}), (\ref{mtpr}), (\ref{t_n_st}) and
(\ref{m_t_pr})):
$$
S_1(n) =2^{-\hat m_0(s_*+1/q-1/p_1)} n^{-1/2} \cdot 2^{\hat
m_0/p_1'} \stackrel{(\ref{hat_mt})}{=} n^{-s_*-\frac 1q +\frac
12},
$$
$$
S_2(n) =2^{-m_0^{(p_1')}(s_*+1/q-1/p_1)} n^{-1/2} \cdot
2^{m_0^{(p_1')}/p_1'} \stackrel{(\ref{line_mt})}{=}
n^{-(s_*+1/q-1/p_1)p_1'/2},
$$
$$
S_3(n) = 2^{\mu_*k_*t_*(n)}\cdot 2^{-\hat
m_{t_*(n)}(s_*+1/q-1/p_1)} n^{-\frac 12}\cdot
2^{\gamma_*k_*t_*(n)/p'_1}\cdot 2^{\hat
m_{t_*(n)}/p'_1}\stackrel{(\ref{hat_mt})}{\underset{\mathfrak{Z}_0}{\asymp}}
$$$$\asymp 2^{\mu_*k_*t_*(n)}\cdot 2^{-\hat m_{t_*(n)}(s_*+1/q-1/p_1)}
n^{-\frac 12 +\frac{1}{p_1'}} \stackrel{(\ref{hat_mt}),
(\ref{tst}), (\ref{eq})}{\underset{\mathfrak{Z}_0}{\asymp}}
n^{-\hat \theta -(1/p_0-1/p_1)(1-\tilde \theta/s_*)+1/2-1/p_1},
$$
$$
S_4(n) = 2^{\mu_*k_*t_{**}(n)}\cdot 2^{-
m^{(p_1')}_{t_{**}(n)}(s_*+1/q-1/p_1)} n^{-1/2} \cdot
2^{\gamma_*k_*t_{**}(n)/p_1'}\cdot 2^{m^{(p_1')}_{t_{**}(n)}/p_1'}
\stackrel{(\ref{line_mt}),(\ref{tsst}),
(\ref{eq})}{\underset{\mathfrak{Z}_0}{\asymp}}
$$
$$
\asymp n^{-\left(\hat \theta +(1/p_0-1/p_1)(1-\tilde
\theta/s_*)\right)p_1'/2},
$$
$$
S_5(n) = 2^{\mu_*k_*t_{***}(n)}\cdot
2^{-m^{(p_0')}_{t_{***}(n)}(s_*+1/q-1/p_1)}
\stackrel{(\ref{line_mt}),
(\ref{tssst})}{\underset{\mathfrak{Z}_0}{\asymp}} n^{-p_0'\hat
\theta/2}.
$$

By the hypotheses of the theorem, there is $j_*\in \{1, \, \dots,
\, j_0\}$ such that $\theta _{j_*}<\min _{j\ne j_*} \theta_j$. We
again take into account (\ref{til_mt_t}), (\ref{mtpr}),
(\ref{t_n_st}) and (\ref{m_t_pr}), appropriately choose $t_1(n)$,
$m_1(n)$ and $\varepsilon>0$, and obtain
$S'\underset{\mathfrak{Z}_0} {\lesssim} S_1(n) + S_2(n) + S_3(n) +
S_4(n) + S_5(n)$; this gives the desired upper estimate.

{\bf The case $q>2$, $\max\{p_0, \, p_1\}=2$,
$\frac{1}{p_0}+\frac{1}{q}\le 1$, $\frac{1}{p_1} +\frac{1}{q}\le
1$, $p_1> p_0$} is regarded as the previous case, taking into
account that $\hat m_t=m^{(p_1')}_t$. Hence, we obtain the
estimate $S\underset{\mathfrak{Z}_0} {\lesssim} S_1(n) + S_2(n) +
S_3(n)$, where
$$
S_1(n)\underset{\mathfrak{Z}_0}{\asymp} n^{-s_*-\frac 1q +\frac
12}, \quad S_2(n)\underset{\mathfrak{Z}_0}{\asymp} n^{-\hat \theta
-(1/p_0-1/2)(1-\tilde \theta/s_*)}, \quad S_3(n)
\underset{\mathfrak{Z}_0}{\asymp}n^{-p_0'\hat \theta/2}.
$$

{\bf The case $q>2$, $\max\{p_0, \, p_1\}<2$,
$\frac{1}{p_0}+\frac{1}{q}\le 1$, $\frac{1}{p_1} +\frac{1}{q}\le
1$, $p_1<p_0$.} We define the numbers $m_t'$ by the equation
\begin{align}
\label{mtpr1} 2^{\mu_*k_*t}\cdot 2^{-m_t'(s_*+1/q-1/p_1)}\cdot
n^{-1/2}\cdot 2^{\gamma_* k_*t/p_1'}\cdot 2^{m_t'/p_1'} =
2^{-\alpha_*k_*t}\cdot 2^{-m_t'(1/q-1/p_0)};
\end{align}
the numbers $t_*(n)$, $t_{**}(n)$, $t_{***}(n)$ are defined by the
equations
\begin{align}
\label{t_n_st1} \tilde m_{t_*(n)} = \hat m_{t_*(n)}, \quad \tilde
m_{t_{**}(n)} = m_{t_{**}(n)}^{(p_0')}, \quad m_{t_{***}(n)} =
m_{t_{***}(n)}^{(p_1')}.
\end{align}
Then
\begin{align}
\label{m_t_pr1} m'_{t_{**}(n)} \stackrel{(\ref{til_mt_t}),
(\ref{line_mt})}{=} \tilde m_{t_{**}(n)}, \quad m'_{t_{***}(n)}
\stackrel{(\ref{mt_t}), (\ref{line_mt})}{=} m_{t_{***}(n)},
\end{align}
\begin{align}
\label{tst1}
2^{(\mu_*+\alpha_*+\gamma_*(s_*+1/p_0-1/p_1))k_*t_*(n)}
\stackrel{(\ref{hat_mt}),(\ref{til_mt_t})}{=} n^{s_*},
\end{align}
\begin{align}
\label{tsst1}
2^{(\mu_*+\alpha_*+\gamma_*(s_*+1/p_0-1/p_1))k_*t_{**}(n)}
\stackrel{(\ref{til_mt_t}),(\ref{line_mt})}{=} n^{s_*p_0'/2},
\end{align}
\begin{align}
\label{tssst1}
2^{(\mu_*+\alpha_*+\gamma_*(s_*+1/p_0-1/p_1))k_*t_{***}(n)}
\stackrel{(\ref{mt_t}),(\ref{line_mt})}{=}
n^{(s_*+1/p_0-1/p_1)p_1'/2}.
\end{align}

Now we estimate the sum (\ref{sum}).

Let
$$
{\rm I} = \{(t, \, m):\; 0\le t\le t_{***}(n), \; m_t^*\le m\le
m_t^{(p_1')}, \; m\ge \tilde m_t, \; m\ge m_t'\},
$$
$$
{\rm II} = \{(t, \, m):\; t_*(n)\le t_{***}(n), \; m_t^*\le m\le
m_t^{(p_0')}, \; m\le \tilde m_t\},
$$
$$
{\rm III} =\{(t, \, m):\; t_{**}(n)<t\le t_{***}(n), \;
m_t^{(p_0')}\le m\le m_t'\},
$$
$$
{\rm IV} = \{(t, \, m):\; 0\le t\le t_{***}(n), \; m_t^{(p_1')}\le
m<\infty\}.
$$
In ${\rm I}$ and ${\rm IV}$ we apply the inclusion $W_{t,m}\subset
2^{\mu_*k_*t}\cdot 2^{-m(s_*+1/q-1/p_1)}B^{\nu_{t,m}}_{p_1}$, in
${\rm II}$ and ${\rm III}$ we use the inclusion $W_{t,m} \subset
2^{-\alpha_*k_*t}\cdot 2^{-m(1/q-1/p_0)}B^{\nu_{t,m}}_{p_0}$.

Applying Theorem \ref{gl_teor_lin}, (\ref{p0lq_emb}) and
(\ref{tssst1}), we get
$$
S \underset{\mathfrak{Z}_0}{\lesssim} n^{-\hat \theta p_1'/2}+\sum
\limits _{(t, \, m)\in {\rm I}} 2^{\mu_*k_*t}\cdot
2^{-m(s_*+1/q-1/p_1)} l_{t,m}^{-1/2} \cdot
2^{\gamma_*k_*t/p_1'}\cdot 2^{m/p_1'}+
$$
$$
+ \sum \limits _{(t, \, m)\in {\rm II}} 2^{-\alpha_*k_*t}\cdot
2^{-m(1/q-1/p_0)}l_{t,m}^{-1/2} \cdot 2^{\gamma_*k_*t/p_0'}\cdot
2^{m/p_0'}+ \sum \limits _{(t, \, m)\in {\rm III}}
2^{-\alpha_*k_*t}\cdot 2^{-m(1/q-1/p_0)}+$$
$$
+\sum \limits _{(t, \, m)\in {\rm IV}} 2^{\mu_*k_*t}\cdot
2^{-m(s_*+1/q-1/p_1)}=:S'.
$$
In the second and the third sums, there is an increasing geometric
progression in $m$, in the last sum, there is a decreasing
geometric progression in $m$.

We also notice that
\begin{align}
\label{eq1} \begin{array}{c}
-s_*\frac{\alpha_*+\gamma_*/p_0-\gamma_*/q}{\mu_* +\alpha_*
+\gamma_*(s_*+1/p_0-1/p_1)} -\frac 1q +\frac{1}{p_0}= \\
= -\hat
\theta-\left(\frac{1}{p_1}-\frac{1}{p_0}\right)\frac{\tilde
\theta}{s_*}.
\end{array}
\end{align}

By the hypotheses of the theorem, there is $j_*\in \{1, \, \dots,
\, j_0\}$ such that $\theta _{j_*}<\min _{j\ne j_*}\theta_j$.
Taking into account (\ref{til_mt_t}) and (\ref{mtpr1}) together
with (\ref{t_n_st1}) and (\ref{m_t_pr1}), for appropriately chosen
$t_1(n)$, $m_1(n)$, we obtain $S'\underset{\mathfrak{Z}_0}
{\lesssim} S_1(n) + S_2(n) + S_3(n) + S_4(n) + S_5(n)$, where
$$
S_1(n) =2^{-\hat m_0(s_*+1/q-1/p_1)} n^{-1/2} \cdot 2^{\hat
m_0/p_1'} \stackrel{(\ref{hat_mt})}{=} n^{-s_*-\frac 1q +\frac
12},
$$
$$
S_2(n) =2^{-m_0^{(p_1')}(s_*+1/q-1/p_1)} n^{-1/2} \cdot
2^{m_0^{(p_1')}/p_1'} \stackrel{(\ref{line_mt})}{=}
n^{-(s_*+1/q-1/p_1)p_1'/2},
$$
$$
S_3(n) = 2^{-\alpha_*k_*t_*(n)}\cdot 2^{-\hat
m_{t_*(n)}(1/q-1/p_0)}n^{-1/2} \cdot
2^{\gamma_*k_*t_*(n)/p_0'}\cdot 2^{\hat m_{t_*(n)}/p_0'}
\stackrel{(\ref{hat_mt})}{\underset{\mathfrak{Z}_0}{\asymp}}
$$
$$
\asymp 2^{-(\alpha_*+\gamma_*/p_0-\gamma_*/q)k_*t_*(n)} \cdot
n^{-1/q+1/p_0}\cdot n^{-1/2+1/p_0'}
\stackrel{(\ref{tst1}),(\ref{eq1})}{\underset{\mathfrak{Z}_0}{\asymp}}
n^{-\hat \theta -(1/p_1-1/p_0)\tilde \theta/s_* +1/2-1/p_0},
$$
$$
S_4(n) =2^{-\alpha_*k_*t_{**}(n)}\cdot
2^{-m^{(p_0')}_{t_{**}(n)}(1/q-1/p_0)}n^{-1/2} \cdot
2^{\gamma_*k_*t_{**}(n)/p_0'}\cdot 2^{m^{(p_0')}_{t_{**}(n)}/p_0'}
\stackrel{(\ref{line_mt})}{\underset{\mathfrak{Z}_0}{\asymp}}
$$
$$
\asymp 2^{-(\alpha_*+\gamma_*/p_0-\gamma_*/q)k_*t_{**}(n)} \cdot
n^{-(1/q-1/p_0)p_0'/2} \stackrel{(\ref{tsst1}),
(\ref{eq1})}{\underset{\mathfrak{Z}_0}{\asymp}} n^{-(\hat \theta
+(1/p_1-1/p_0)\tilde \theta/s_*)p_0'/2},
$$
$$
S_5(n) = 2^{\mu_*k_*t_{***}(n)}\cdot
2^{-m^{(p_1')}_{t_{***}(n)}(s_*+1/q-1/p_1)} n^{-1/2} \cdot
2^{\gamma_*k_*t_{***}(n)/p_1'}\cdot
2^{m_{t_{***}(n)}^{(p_1')}/p_1'}
\stackrel{(\ref{line_mt})}{\underset{\mathfrak{Z}_0}{\asymp}}
$$
$$
\asymp 2^{(\mu_*+\gamma_*(s_*+1/q-1/p_1))t_{***}(n)}\cdot
n^{-(s_*+1/q-1/p_1)p_1'/2}\stackrel{(\ref{tssst1})}{\underset{\mathfrak{Z}_0}{\asymp}}
n^{-\hat \theta p_1'/2}.
$$

{\bf The case $q>2$, $\max\{p_0, \, p_1\}=2$,
$\frac{1}{p_0}+\frac{1}{q}\le 1$, $\frac{1}{p_1} +\frac{1}{q}\le
1$, $p_1<p_0$} is regarded as the previous case; here we take into
account that $\hat m_t=m_t^{(p_0')}$. Hence, we obtain the
estimate $S \underset{\mathfrak{Z}_0}{\lesssim} S_1(n)+ S_2(n) +
S_3(n)+ S_4(n)$, where
$$
S_1(n)\underset{\mathfrak{Z}_0}{\asymp} n^{-s_*-\frac 1q +\frac
12}, \quad S_2(n)\underset{\mathfrak{Z}_0}{\asymp}
n^{-(s_*+1/q-1/p_1)p_1'/2},
$$
$$
S_3(n)\underset{\mathfrak{Z}_0}{\asymp} n^{-\hat \theta
-(1/p_1-1/2)\tilde \theta/s_*}, \quad
S_4(n)\underset{\mathfrak{Z}_0}{\asymp} n^{-\hat \theta p_1'/2}.
$$

{\bf The case $q>2$, $\max\{p_0, \, p_1\}<2$,
$\frac{1}{p_0}+\frac{1}{q}\le 1$, $\frac{1}{p_1} +\frac{1}{q}\le
1$, $p_1=p_0$.} We denote $p=p_0=p_1$. Arguing as in the previous
cases and taking into account that $m_t^{(p_0')}=m_t^{(p_1')}$, we
obtain the estimate $S \underset{\mathfrak{Z}_0}{\lesssim} S_1(n)+
S_2(n) + S_3(n)+ S_4(n)$, where
$$
S_1(n)\underset{\mathfrak{Z}_0}{\asymp}n^{-s_*-\frac 1q +\frac
12}, \quad S_2(n)\underset{\mathfrak{Z}_0}{\asymp}
n^{-(s_*+1/q-1/p)p'/2},
$$
$$
S_3(n)\underset{\mathfrak{Z}_0}{\asymp} n^{-\hat \theta +
1/2-1/p}, \quad S_4(n) \underset{\mathfrak{Z}_0}{\asymp} n^{-\hat
\theta p'/2}.
$$

{\bf The case $q>2$, $\max\{p_0, \, p_1\}<2$,
$\frac{1}{p_0}+\frac{1}{q}> 1$, $\frac{1}{p_1} +\frac{1}{q}< 1$.}
Then $\frac{1}{p_1}< \frac{1}{q'} < \frac{1}{p_0}$. Let $\lambda
\in (0, \, 1)$ be such that
\begin{align}
\label{lam} \frac{1}{q'} =\frac{1-\lambda}{p_1}
+\frac{\lambda}{p_0}.
\end{align}
Then
\begin{align}
\label{nepr1} \begin{array}{c} 2^{\mu_*k_*t}\cdot 2^{-\tilde
m_t(s_*+1/q-1/p_1)} n^{-1/2} \cdot 2^{\gamma_*k_*t/p_1'}\cdot
2^{\tilde m_t/p_1'} \stackrel{(\ref{til_mt_t})}{=} \\
= 2^{((1-\lambda)\mu_*-\lambda \alpha_*)k_*t}\cdot 2^{-\tilde
m_t((1-\lambda)(s_*+1/q-1/p_1)+\lambda(1/q-1/p_0))} n^{-1/2}\cdot
2^{\gamma_*k_*t/q}\cdot 2^{\tilde m_t/q},
\end{array}
\end{align}
\begin{align}
\label{nepr2} \begin{array}{c} 2^{-\alpha_*k_*t}\cdot 2^{-
m_t(1/q-1/p_0)} \stackrel{(\ref{mt_t})}{=} \\
= 2^{((1-\lambda)\mu_*-\lambda \alpha_*)k_*t}\cdot
2^{-m_t((1-\lambda)(s_*+1/q-1/p_1)+\lambda(1/q-1/p_0))}.
\end{array}
\end{align}

We define the numbers $m_t'$ by the equation
\begin{align}
\label{mtpr2} \begin{array}{c}2^{\mu_*k_*t}\cdot
2^{-m'_t(s_*+1/q-1/p_1)} =\\= 2^{((1-\lambda)\mu_*-\lambda
\alpha_*)k_*t}\cdot
2^{-m'_t((1-\lambda)(s_*+1/q-1/p_1)+\lambda(1/q-1/p_0))}
n^{-1/2}\cdot 2^{\gamma_*k_*t/q}\cdot 2^{m'_t/q}; \end{array}
\end{align}
the numbers $t_*(n)$, $t_{**}(n)$, $t_{***}(n)$, $t(n)$ are
defined by the equations
\begin{align}
\label{t_n_st2} \tilde m_{t_*(n)} = \hat m_{t_*(n)}, \quad \tilde
m_{t_{**}(n)} = m_{t_{**}(n)}^{(p_1')}, \quad m_{t_{***}(n)} =
m_{t_{***}(n)}^{(q)}, \quad m_{t(n)} = \hat m_{t(n)}.
\end{align}
Now using (\ref{til_mt_t}), (\ref{mt_t}), (\ref{line_mt}),
(\ref{lam}), (\ref{mtpr2}), (\ref{t_n_st2}) we get
\begin{align}
\label{m_t_pr2} m'_{t_{**}(n)} = \tilde m_{t_{**}(n)}, \quad
m'_{t_{***}(n)} = m_{t_{***}(n)},
\end{align}
\begin{align}
\label{tst2}
2^{(\mu_*+\alpha_*+\gamma_*(s_*+1/p_0-1/p_1))k_*t_*(n)} = n^{s_*},
\end{align}
\begin{align}
\label{tsst2}
2^{(\mu_*+\alpha_*+\gamma_*(s_*+1/p_0-1/p_1))k_*t_{**}(n)} =
n^{s_*p_1'/2},
\end{align}
\begin{align}
\label{tssst2}
2^{(\mu_*+\alpha_*+\gamma_*(s_*+1/p_0-1/p_1))k_*t_{***}(n)} =
n^{(s_*+1/p_0-1/p_1)q/2},
\end{align}
\begin{align}
\label{tt2} 2^{(\mu_*+\alpha_*+\gamma_*(s_*+1/p_0-1/p_1))k_*t(n)}
= n^{s_*+1/p_0-1/p_1}.
\end{align}

Let us estimate the sum (\ref{sum}).

We set
$$
{\rm I} = \{(t, \, m):\; 0\le t\le t_{***}(n), \; m_t^*\le m\le
m_t^{(p_1')}, \; m\ge \tilde m_t\},
$$
$$
{\rm II} = \{(t, \, m):\; 0\le t\le t_{***}(n), \; m_t^*\le m\le
\tilde m_t, \; m\ge m_t, \; m\le m_t'\},
$$
$$
{\rm III} = \{(t, \, m):\; 0\le t\le t_{***}(n), \; m\ge
m_t^{(p_1')}, \; m\ge m_t'\},
$$
$$
{\rm IV} = \{(t, \, m):\; 0\le t\le t_{***}(n), \; m\ge m_t^*, \;
m\le m_t\}.
$$
In ${\rm I}$ and ${\rm III}$ we use the inclusion $W_{t,m}\subset
2^{\mu_*k_*t}\cdot 2^{-m(s_*+1/q-1/p_1)}B_{p_1}^{\nu_{t,m}}$, in
${\rm II}$ we apply (\ref{wkm_lq}), in ${\rm IV}$ we use the
inclusion $W_{t,m} \subset 2^{-\alpha_*k_*t}\cdot
2^{-m(1/q-1/p_0)} B_{p_0}^{\nu_{t,m}}$.

Applying Theorem \ref{gl_teor_lin} and taking into account
(\ref{p0lq_emb}), (\ref{tssst2}), we get
$$
S \underset{\mathfrak{Z}_0}{\lesssim} n^{-\hat \theta q/2}+\sum
\limits _{(t, \, m)\in {\rm I}} 2^{\mu_*k_*t}\cdot
2^{-m(s_*+1/q-1/p_1)} l_{t,m}^{-1/2} \cdot
2^{\gamma_*k_*t/p_1'}\cdot 2^{m/p_1'}+
$$
$$
+\sum \limits _{(t, \, m)\in {\rm II}} 2^{((1-\lambda)\mu_*
-\lambda\alpha_*)k_*t}\cdot 2^{-m((1-\lambda) (s_*+1/q-1/p_1)
+\lambda(1/q-1/p_0))}l_{t,m}^{-1/2} \cdot 2^{\gamma_*k_*t/q}\cdot
2^{m/q}+
$$
$$
+\sum \limits _{(t, \, m)\in {\rm III}} 2^{\mu_*k_*t}\cdot
2^{-m(s_*+1/q-1/p_1)} + \sum \limits _{(t, \, m)\in {\rm IV}}
2^{-\alpha_*k_*t}\cdot 2^{-m(1/q-1/p_0)} l_{t,m}^{-1/2} \cdot
2^{\gamma_*k_*t/q}\cdot 2^{m/q}=:S'.
$$

By the hypotheses of the theorem, $\theta_{j_*}<\min _{j\ne
j_*}\theta_j$. In addition, (\ref{nepr1}), (\ref{nepr2}),
(\ref{mtpr2}), (\ref{t_n_st2}), (\ref{m_t_pr2}) hold. Hence, for
appropriate $t_1(n)$, $m_1(n)$ the estimate
$S'\underset{\mathfrak{Z}_0} {\lesssim} S_1(n) + S_2(n) + S_3(n) +
S_4(n) + S_5(n)+S_6(n)$ holds with
$$
S_1(n) =2^{-\hat m_0(s_*+1/q-1/p_1)} n^{-1/2} \cdot 2^{\hat
m_0/p_1'} \stackrel{(\ref{hat_mt})}{=} n^{-s_*-\frac 1q +\frac
12},
$$
$$
S_2(n) =2^{-m_0^{(p_1')}(s_*+1/q-1/p_1)} n^{-1/2} \cdot
2^{m_0^{(p_1')}/p_1'} \stackrel{(\ref{line_mt})}{=}
n^{-(s_*+1/q-1/p_1)p_1'/2},
$$
$$
S_3(n) = 2^{\mu_*k_*t_*(n)}\cdot 2^{-\hat
m_{t_*(n)}(s_*+1/q-1/p_1)} n^{-\frac 12}\cdot
2^{\gamma_*k_*t_*(n)/p_1'}\cdot 2^{\hat m_{t_*(n)}/p_1'}
\stackrel{(\ref{hat_mt}), (\ref{eq}),
(\ref{tst2})}{\underset{\mathfrak{Z}_0}{\asymp}}$$$$\asymp
n^{-\hat \theta -(1/p_0-1/p_1)(1-\tilde \theta/s_*)+1/2-1/p_1},
$$
$$
S_4(n) = 2^{\mu_*k_*t_{**}(n)}\cdot
2^{-m^{(p_1')}_{t_{**}(n)}(s_*+1/q-1/p_1)} n^{-1/2} \cdot
2^{\gamma_*k_*t_{**}(n)/p_1'}\cdot 2^{m^{(p_1')}_{t_{**}(n)}/p_1'}
\stackrel{(\ref{line_mt}), (\ref{eq}),
(\ref{tsst2})}{\underset{\mathfrak{Z}_0}{\asymp}}
$$
$$
\asymp n^{-\left(\hat \theta +(1/p_0-1/p_1)(1-\tilde
\theta/s_*)\right)p_1'/2},
$$
$$
S_5(n)= 2^{-\alpha_*k_*t(n)}\cdot 2^{-\hat m_{t(n)}(1/q-1/p_0)}
n^{-1/2} \cdot 2^{\gamma_*k_*t(n)/q}\cdot 2^{\hat m_{t(n)}/q}
\stackrel{(\ref{hat_mt})}{\underset{\mathfrak{Z}_0}{\asymp}}
$$
$$
\asymp 2^{-(\alpha_*+\gamma_*/p_0-\gamma_*/q)t(n)} \cdot
n^{-1/q+1/p_0}\cdot n^{1/q-1/2}
\stackrel{(\ref{tt2})}{\underset{\mathfrak{Z}_0}{\asymp}} n^{-\hat
\theta -1/2+1/q},
$$
$$
S_6(n)= 2^{-\alpha_*k_*t_{***}(n)}\cdot
2^{-m^{(q)}_{t_{***}(n)}(1/q-1/p_0)} n^{-1/2} \cdot
2^{\gamma_*k_*t_{***}(n)/q}\cdot 2^{m^{(q)}_{t_{***}(n)}/q}
\stackrel{(\ref{line_mt})}{\underset{\mathfrak{Z}_0}{\asymp}}
$$
$$
\asymp 2^{-(\alpha_*+\gamma_*/p_0-\gamma_*/q)t_{***}(n)} \cdot
n^{-(1/q-1/p_0)q/2}
\stackrel{(\ref{tssst2})}{\underset{\mathfrak{Z}_0}{\asymp}}
n^{-q\hat \theta/2}.
$$

{\bf The case $q>2$, $\max\{p_0, \, p_1\}=2$,
$\frac{1}{p_0}+\frac{1}{q}> 1$, $\frac{1}{p_1} +\frac{1}{q}< 1$.}
Here we argue as in the previous case, taking into account that
$m_t^{(p_1')} =\hat m_t$. Hence, we get the estimate
$S\underset{\mathfrak{Z}_0} {\lesssim} S_1(n) + S_2(n) + S_3(n) +
S_4(n)$ with
$$
S_1(n) \underset{\mathfrak{Z}_0}{\asymp} n^{-s_*-\frac 1q +\frac
12}, \quad S_2(n) \underset{\mathfrak{Z}_0}{\asymp} n^{-\hat
\theta -(1/p_0-1/2)(1-\tilde \theta/s_*)},
$$
$$
S_3(n) \underset{\mathfrak{Z}_0}{\asymp} n^{-\hat \theta
-1/2+1/q}, \quad S_4(n) \underset{\mathfrak{Z}_0}{\asymp}
n^{-q\hat \theta/2}.
$$

{\bf The case $q>2$, $\max\{p_0, \, p_1\}<2$,
$\frac{1}{p_0}+\frac{1}{q}< 1$, $\frac{1}{p_1} +\frac{1}{q}> 1$.}
Then $\frac{1}{p_0}< \frac{1}{q'} < \frac{1}{p_1}$. Let $\lambda
\in (0, \, 1)$ be given by (\ref{lam}). Then
\begin{align}
\label{nepr3} \begin{array}{c} 2^{\mu_*k_*t}\cdot
2^{-m_t(s_*+1/q-1/p_1)} \stackrel{(\ref{mt_t})}{=} \\ =
2^{((1-\lambda)\mu_*-\lambda \alpha_*)k_*t}\cdot
2^{-m_t((1-\lambda)(s_*+1/q-1/p_1)+\lambda(1/q-1/p_0))},\end{array}
\end{align}
\begin{align}
\label{nepr4} \begin{array}{c} 2^{-\alpha_*k_*t}\cdot 2^{-\tilde
m_t(1/q-1/p_0)} n^{-1/2} \cdot 2^{\gamma_*k_*t/p_0'}\cdot
2^{\tilde m_t/p_0'} \stackrel{(\ref{til_mt_t})}{=} \\ =
2^{((1-\lambda)\mu_*-\lambda \alpha_*)k_*t}\cdot 2^{-\tilde
m_t((1-\lambda)(s_*+1/q-1/p_1)+\lambda(1/q-1/p_0))}n^{-1/2}\cdot
2^{\gamma_*k_*t/q}\cdot 2^{\tilde m_t/q}.\end{array}
\end{align}

We define the numbers $m_t'$ by the equation
\begin{align}
\label{mtpr3} \begin{array}{c}2^{-\alpha_*k_*t}\cdot
2^{-m'_t(1/q-1/p_0)} =\\= 2^{((1-\lambda)\mu_*-\lambda
\alpha_*)k_*t}\cdot
2^{-m'_t((1-\lambda)(s_*+1/q-1/p_1)+\lambda(1/q-1/p_0))}
n^{-1/2}\cdot 2^{\gamma_*k_*t/q}\cdot 2^{m'_t/q}; \end{array}
\end{align}
the numbers $t_*(n)$, $t_{**}(n)$, $t_{***}(n)$, $t(n)$ are
defined by the equations
\begin{align}
\label{t_n_st3} \tilde m_{t_*(n)} = \hat m_{t_*(n)}, \quad \tilde
m_{t_{**}(n)} = m_{t_{**}(n)}^{(p_0')}, \quad m_{t_{***}(n)} =
m_{t_{***}(n)}^{(q)}, \quad m_{t(n)} = \hat m_{t(n)}.
\end{align}
Using (\ref{til_mt_t}), (\ref{mt_t}), (\ref{line_mt}),
(\ref{lam}), (\ref{mtpr3}), (\ref{t_n_st3}) we get
\begin{align}
\label{m_t_pr3} m'_{t_{**}(n)} = \tilde m_{t_{**}(n)}, \quad
m'_{t_{***}(n)} = m_{t_{***}(n)},
\end{align}
\begin{align}
\label{tst3}
2^{(\mu_*+\alpha_*+\gamma_*(s_*+1/p_0-1/p_1))k_*t_*(n)} = n^{s_*},
\end{align}
\begin{align}
\label{tsst3}
2^{(\mu_*+\alpha_*+\gamma_*(s_*+1/p_0-1/p_1))k_*t_{**}(n)} =
n^{s_*p_0'/2},
\end{align}
\begin{align}
\label{tssst3}
2^{(\mu_*+\alpha_*+\gamma_*(s_*+1/p_0-1/p_1))k_*t_{***}(n)} =
n^{(s_*+1/p_0-1/p_1)q/2},
\end{align}
\begin{align}
\label{tt3} 2^{(\mu_*+\alpha_*+\gamma_*(s_*+1/p_0-1/p_1))k_*t(n)}
= n^{s_*+1/p_0-1/p_1}.
\end{align}

Now we estimate the sum (\ref{sum}).

Let
$$
{\rm I} = \{(t, \, m):\; 0\le t\le t_{***}(n), \; m_t^*\le m\le
m_t^{(q)}, \; m\ge m_t\},
$$
$$
{\rm II} = \{(t, \, m):\; 0\le t\le t_{***}(n), \; m\ge m_t^*, \;
\tilde m_t\le m\le m_t, \; m\ge m_t'\},
$$
$$
{\rm III} = \{(t, \, m):\; 0\le t\le t_{***}(n), \; m\ge
m_t^{(q)}\},
$$
$$
{\rm IV} = \{(t, \, m):\; 0\le t\le t_{***}(n), \; m_t^*\le m\le
m_t^{(p_0')}, \; m\le \tilde m_t\},
$$
$$
{\rm V} = \{(t, \, m):\; 0\le t\le t_{***}(n), \; m\ge
m_t^{(p_0')}, \; m\le m_t'\}.
$$

In ${\rm I}$ and ${\rm III}$ we apply the inclusion
$W_{t,m}\subset 2^{\mu_*k_*t}\cdot
2^{-m(s_*+1/q-1/p_1)}B_{p_1}^{\nu_{t,m}}$, in ${\rm II}$ we use
(\ref{wkm_lq}), in ${\rm IV}$ and ${\rm V}$ we use the inclusion
$W_{t,m} \subset 2^{-\alpha_*k_*t}\cdot 2^{-m(1/q-1/p_0)}
B_{p_0}^{\nu_{t,m}}$.

Applying (\ref{p0lq_emb}), (\ref{tssst3}) and Theorem
\ref{gl_teor_lin}, we get
$$
S \underset{\mathfrak{Z}_0}{\lesssim} n^{-\hat \theta q/2}+\sum
\limits _{(t, \, m)\in {\rm I}} 2^{\mu_*k_*t}\cdot
2^{-m(s_*+1/q-1/p_1)} l_{t,m}^{-1/2} \cdot 2^{\gamma_*k_*t/q}\cdot
2^{m/q}+
$$
$$
+\sum \limits _{(t, \, m)\in {\rm II}} 2^{((1-\lambda)\mu_*
-\lambda\alpha_*)k_*t}\cdot 2^{-m((1-\lambda) (s_*+1/q-1/p_1)
+\lambda(1/q-1/p_0))}l_{t,m}^{-1/2} \cdot 2^{\gamma_*k_*t/q}\cdot
2^{m/q}+
$$
$$
+\sum \limits _{(t, \, m)\in {\rm III}} 2^{\mu_*k_*t}\cdot
2^{-m(s_*+1/q-1/p_1)} + \sum \limits _{(t, \, m)\in {\rm IV}}
2^{-\alpha_*k_*t}\cdot 2^{-m(1/q-1/p_0)} l_{t,m}^{-1/2} \cdot
2^{\gamma_*k_*t/p_0'}\cdot 2^{m/p_0'}+
$$
$$
+\sum \limits _{(t, \, m)\in {\rm V}} 2^{-\alpha_*k_*t}\cdot
2^{-m(1/q-1/p_0)} =:S'.
$$
By the hypotheses of the theorem, $\theta_{j_*}<\min _{j\ne
j_*}\theta_j$. In addition, (\ref{nepr3}), (\ref{nepr4}),
(\ref{mtpr3}), (\ref{t_n_st3}), (\ref{m_t_pr3}) hold. For
appropriate $t_1(n)$, $m_1(n)$ we get the estimate
$S'\underset{\mathfrak{Z}_0} {\lesssim} S_1(n) + S_2(n) + S_3(n) +
S_4(n) + S_5(n)+S_6(n)$, where
$$
S_1(n) =2^{-\hat m_0(s_*+1/q-1/p_1)} n^{-1/2} \cdot 2^{\hat m_0/q}
\stackrel{(\ref{hat_mt})}{=} n^{-s_*-\frac 12 +\frac{1}{p_1}},
$$
$$
S_2(n) =2^{-m_0^{(q)}(s_*+1/q-1/p_1)} n^{-1/2} \cdot
2^{m_0^{(q)}/q} \stackrel{(\ref{line_mt})}{=}
n^{-(s_*+1/q-1/p_1)q/2},
$$
$$
S_3(n) = 2^{-\alpha_*k_*t_*(n)}\cdot 2^{-\hat
m_{t_*(n)}(1/q-1/p_0)}n^{-1/2} \cdot
2^{\gamma_*k_*t_*(n)/p_0'}\cdot 2^{\hat m_{t_*(n)}/p_0'}
\stackrel{(\ref{hat_mt})}{\underset{\mathfrak{Z}_0}{\asymp}}
$$
$$
\asymp 2^{-(\alpha_*+\gamma_*/p_0-\gamma_*/q)k_*t_*(n)} \cdot
n^{-1/q+1/p_0}\cdot n^{-1/2+1/p_0'}
\stackrel{(\ref{eq1}),(\ref{tst3})
}{\underset{\mathfrak{Z}_0}{\asymp}} n^{-\hat \theta
-(1/p_1-1/p_0)\tilde \theta/s_* +1/2-1/p_0},
$$
$$
S_4(n) =2^{-\alpha_*k_*t_{**}(n)}\cdot
2^{-m^{(p_0')}_{t_{**}(n)}(1/q-1/p_0)}n^{-1/2} \cdot
2^{\gamma_*k_*t_{**}(n)/p_0'}\cdot 2^{m^{(p_0')}_{t_{**}(n)}/p_0'}
\stackrel{(\ref{line_mt})}{\underset{\mathfrak{Z}_0}{\asymp}}
$$
$$
\asymp 2^{-(\alpha_*+\gamma_*/p_0-\gamma_*/q)k_*t_{**}(n)} \cdot
n^{-(1/q-1/p_0)p_0'/2}
\stackrel{(\ref{eq1}),(\ref{tsst3})}{\underset{\mathfrak{Z}_0}{\asymp}}
n^{-(\hat \theta +(1/p_1-1/p_0)\tilde \theta/s_*)p_0'/2},
$$
$$
S_5(n) = 2^{\mu_*k_*t(n)}\cdot 2^{-\hat m_{t(n)}(s_*+1/q-1/p_1)}
n^{-1/2}\cdot 2^{\gamma_*k_*t(n)/q}\cdot 2^{\hat m_{t(n)}/q}
\stackrel{(\ref{hat_mt})}{\underset{\mathfrak{Z}_0}{\asymp}}
$$
$$
\asymp 2^{(\mu_*+\gamma_*(s_*+1/q-1/p_1))k_*t(n)}\cdot
n^{-s_*-1/q+1/p_1}\cdot n^{-1/2+1/q}
\stackrel{(\ref{tt3})}{\underset{\mathfrak{Z}_0}{\asymp}} n^{-\hat
\theta +1/q-1/2},
$$
$$
S_6(n) = 2^{\mu_*k_*t_{***}(n)}\cdot
2^{-m^{(q)}_{t_{***}(n)}(s_*+1/q-1/p_1)} n^{-1/2}\cdot
2^{\gamma_*k_*t_{***}(n)/q}\cdot 2^{m^{(q)}_{t_{***}(n)}/q}
\stackrel{(\ref{line_mt})}{\underset{\mathfrak{Z}_0}{\asymp}}
$$
$$
\asymp 2^{(\mu_*+\gamma_*(s_*+1/q-1/p_1))k_*t_{***}(n)}\cdot
n^{-(s_*+1/q-1/p_1)q/2}
\stackrel{(\ref{tssst3})}{\underset{\mathfrak{Z}_0}{\asymp}}
n^{-q\hat \theta/2}.
$$

{\bf The case $q>2$, $\max\{p_0, \, p_1\}=2$,
$\frac{1}{p_0}+\frac{1}{q}< 1$, $\frac{1}{p_1} +\frac{1}{q}> 1$.}
We argue as in the previous case, taking into account that $\hat
m_t = m^{(p_0')}_t$. Hence, we obtain the estimate
$S\underset{\mathfrak{Z}_0} {\lesssim} S_1(n) + S_2(n) + S_3(n) +
S_4(n) + S_5(n)$, where
$$
S_1(n) \underset{\mathfrak{Z}_0}{\asymp} n^{-s_*-\frac 12
+\frac{1}{p_1}}, \quad S_2(n) \underset{\mathfrak{Z}_0}{\asymp}
n^{-(s_*+1/q-1/p_1)q/2}, \quad S_3(n)
\underset{\mathfrak{Z}_0}{\asymp} n^{-\hat \theta
-(1/p_1-1/2)\tilde \theta/s_*},
$$
$$
S_4(n) \underset{\mathfrak{Z}_0}{\asymp} n^{-\hat \theta
+1/q-1/2}, \quad S_5(n) \underset{\mathfrak{Z}_0}{\asymp}
n^{-q\hat \theta/2}.
$$

{\bf The case $q>2$, $p_1<2<p_0<q$, $\frac{1}{p_1}+\frac{1}{q}\le
1$.} We define $m_t'$ by the equation
\begin{align}
\label{mtpr4} 2^{-\alpha_*k_*t}\cdot 2^{-m_t'(1/q-1/p_0)} =
2^{\mu_*k_*t}\cdot 2^{-m_t'(s_*+1/q-1/p_1)}\cdot n^{-1/2}\cdot
2^{\gamma_*k_*t/p_1'}\cdot 2^{m_t'/p_1'};
\end{align}
$t(n)$, $t_{***}(n)$ are defined by the equations
\begin{align}
\label{tn} \hat m_{t(n)} = m'_{t(n)}, \quad
m^{(p_1')}_{t_{***}(n)} = m'_{t_{***}(n)}.
\end{align}
Hence
\begin{align}
\label{2tsss} 2^{(\mu_*+\alpha_*+\gamma_*(s_*+1/p_0-1/p_1))t(n)}
\stackrel{(\ref{hat_mt})}{=} n^{s_*+1/p_0-1/2},
\end{align}
\begin{align}
\label{2tstar} 2^{(\mu_*+\alpha_*+\gamma_*(s_*+1/p_0-1/p_1))
t_{***}(n)} \stackrel{(\ref{mt_t}),(\ref{line_mt})}{=}
n^{(s_*+1/p_0-1/p_1)p_1'/2}.
\end{align}

Let us estimate the sum (\ref{sum}).

We set
$$
{\rm I} = \{(t, \, m):\; 0\le t\le t_{***}(n), \; m_t^*\le m\le
m_t^{(p_1')}, \; m\ge m_t'\},
$$
$$
{\rm II} =\{(t, \, m):\; 0\le t\le t_{***}(n), \; m\ge
m_t^{(p_1')}\},
$$
$$
{\rm III} =\{(t, \, m):\; 0\le t\le t_{***}(n), \; m_t^*\le m\le
m'_t\}.
$$
In ${\rm I}$ and ${\rm II}$ we use the inclusion $W_{t,m}\subset
2^{\mu_*k_*t}\cdot 2^{-m(s_*+1/q-1/p_1)}B_{p_1}^{\nu_{t,m}}$, in
${\rm III}$ we use the inclusion $W_{t,m} \subset
2^{-\alpha_*k_*t}\cdot 2^{-m(1/q-1/p_0)} B_{p_0}^{\nu_{t,m}}$.

Applying (\ref{p0lq_emb}), (\ref{2tstar}) and Theorem
\ref{gl_teor_lin}, we get
$$
S \underset{\mathfrak{Z}_0}{\lesssim} n^{-\hat \theta p_1'/2}+\sum
\limits _{(t, \, m)\in {\rm I}} 2^{\mu_*k_*t}\cdot
2^{-m(s_*+1/q-1/p_1)} l_{t,m}^{-1/2} \cdot
2^{\gamma_*k_*t/p_1'}\cdot 2^{m/p_1'}+
$$
$$
+\sum \limits _{(t, \, m)\in {\rm II}} 2^{\mu_*k_*t}\cdot
2^{-m(s_*+1/q-1/p_1)} + \sum \limits _{(t, \, m)\in {\rm III}}
2^{-\alpha_*k_*t}\cdot 2^{-m(1/q-1/p_0)} =:S'.
$$
In the second sum, there is a decreasing geometric progression in
$m$, and in the third sum, there is an increasing geometric
progression in $m$.

Notice that
\begin{align}
\label{eq2} \begin{array}{c}
-\frac{(\alpha_*+\gamma_*/p_0-\gamma_*/q)(s_*+1/p_0-1/2)}
{\mu_*+\alpha_*+\gamma_*(s_*+1/p_0-1/p_1)} -\frac 1q
+\frac{1}{p_0} = \\ =-\hat \theta -\left(\frac{1}{p_1}-\frac
12\right)\frac{\tilde \theta}{s_*}.\end{array}
\end{align}

By the hypotheses of the theorem, $\theta_{j_*}<\min _{j\ne
j_*}\theta_j$. In addition, (\ref{mtpr4}) and (\ref{tn}) hold.
Taking into account (\ref{eq2}), for appropriate $t_1(n)$,
$m_1(n)$ we get $S' \underset{\mathfrak{Z}_0} {\lesssim} S_1(n)+
S_2(n) +S_3(n) +S_4(n)$, where
$$
S_1(n) = 2^{-\hat m_0(s_*+1/q-1/p_1)} n^{-1/2} \cdot 2^{\hat
m_0/p_1'} \stackrel{(\ref{hat_mt})}{=} n^{-s_*-\frac 1q +\frac
12},
$$
$$
S_2(n) =2^{-m_0^{(p_1')}(s_*+1/q-1/p_1)} n^{-1/2} \cdot
2^{m_0^{(p_1')}/p_1'} \stackrel{(\ref{line_mt})}{=}
n^{-(s_*+1/q-1/p_1)p_1'/2},
$$
$$
S_3(n) = 2^{-\alpha_*k_*t(n)}\cdot 2^{-\hat m_{t(n)}(1/q-1/p_0)}
\stackrel{(\ref{hat_mt})}{\underset{\mathfrak{Z}_0}{\asymp}}
2^{-(\alpha_*+\gamma_*/p_0-\gamma_*/q)t(n)} \cdot n^{-1/q+1/p_0}
\stackrel{(\ref{2tsss}),(\ref{eq2})}{\underset{\mathfrak{Z}_0}{\asymp}}
n^{-\hat \theta-(1/p_1-1/2)\tilde \theta/s_*},
$$
$$
S_4(n) = 2^{\mu_*k_*t_{***}(n)}\cdot
2^{-m^{(p_1')}_{t_{***}(n)}(s_*+1/q-1/p_1)}
\stackrel{(\ref{line_mt})}{\underset{\mathfrak{Z}_0}{\asymp}}
$$$$\asymp 2^{(\mu_*+ \gamma_*(s_*+1/q-1/p_1))k_*t_{***}(n)}\cdot
n^{-(s_*+1/q-1/p_1)p_1'/2}
\stackrel{(\ref{2tstar})}{\underset{\mathfrak{Z}_0}{\asymp}}
n^{-\hat \theta p_1'/2}.
$$

{\bf The case $q>2$, $p_0<2<p_1<q$, $\frac{1}{p_0}+\frac{1}{q}\le
1$.} We define $m_t'$ by the equation
\begin{align}
\label{mtpr5} 2^{-\alpha_*k_*t}\cdot 2^{-m_t'(1/q-1/p_0)}\cdot
n^{-1/2}\cdot 2^{\gamma_*k_*t/p_0'}\cdot 2^{m_t'/p_0'} =
2^{\mu_*k_*t}\cdot 2^{-m_t'(s_*+1/q-1/p_1)};
\end{align}
$t(n)$, $t_{***}(n)$ are defined by the equations
\begin{align}
\label{tn_0} \hat m_{t(n)} = m'_{t(n)}, \quad
m^{(p_0')}_{t_{***}(n)} = m'_{t_{***}(n)}.
\end{align}
Hence
\begin{align}
\label{2tsss0} 2^{(\mu_*+\alpha_*+\gamma_*(s_*+1/p_0-1/p_1))t(n)}
\stackrel{(\ref{hat_mt})}{=} n^{s_*+1/2-1/p_1},
\end{align}
\begin{align}
\label{2tstar0} 2^{(\mu_*+\alpha_*+\gamma_*(s_*+1/p_0-1/p_1))
t_{***}(n)} \stackrel{(\ref{mt_t}), (\ref{line_mt})}{=}
n^{(s_*+1/p_0-1/p_1)p_0'/2}.
\end{align}

We estimate the sum (\ref{sum}).

Let
$$
{\rm I} = \{(t, \, m):\; 0\le t\le t_{***}(n), \; m_t^*\le m\le
m_t^{(p_0')}, \; m\ge m_t'\},
$$
$$
{\rm II} =\{(t, \, m):\; 0\le t\le t_{***}(n), \; m\ge
m_t^{(p_0')}\},
$$
$$
{\rm III} =\{(t, \, m):\; 0\le t\le t_{***}(n), \; m_t^*\le m\le
m'_t\}.
$$
In ${\rm I}$ and ${\rm II}$ we use the inclusion $W_{t,m}\subset
2^{\mu_*k_*t}\cdot 2^{-m(s_*+1/q-1/p_1)}B_{p_1}^{\nu_{t,m}}$, and
in ${\rm III}$ we use the inclusion $W_{t,m} \subset
2^{-\alpha_*k_*t}\cdot 2^{-m(1/q-1/p_0)} B_{p_0}^{\nu_{t,m}}$.

Applying (\ref{p0lq_emb}), (\ref{2tstar0}) and Theorem
\ref{gl_teor_lin}, we get
$$
S \underset{\mathfrak{Z}_0}{\lesssim} n^{-\hat \theta p_0'/2}+\sum
\limits _{(t, \, m)\in {\rm I}\cup {\rm II}} 2^{\mu_*k_*t}\cdot
2^{-m(s_*+1/q-1/p_1)}+
$$
$$
+ \sum \limits _{(t, \, m)\in {\rm III}} 2^{-\alpha_*k_*t}\cdot
2^{-m(1/q-1/p_0)} l_{t,m}^{-1/2} \cdot 2^{\gamma_*k_*t/p_0'}\cdot
2^{m/p_0'} =:S'.
$$
In the first sum, there is a decreasing geometric progression in
$m$, in the second sum, there is an increasing geometric
progression in $m$.

Notice that
\begin{align}
\label{eq3} \begin{array}{c}
\frac{(\mu_*+\gamma_*(s_*+1/q-1/p_1))(s_*+1/2-1/p_1)}
{\mu_*+\alpha_*+\gamma_*(s_*+1/p_0-1/p_1)} -s_*-\frac 1q
+\frac{1}{p_1} = \\ =-\hat \theta -\left(\frac{1}{p_0}-\frac
12\right)\left(1-\frac{\tilde \theta}{s_*}\right).\end{array}
\end{align}

By the hypotheses of the theorem, $\theta_{j_*}<\min _{j\ne
j_*}\theta_j$. In addition, (\ref{mtpr5}) and (\ref{tn_0}) hold.
Taking into account (\ref{eq3}), for appropriate $t_1(n)$,
$m_1(n)$ we obtain $S'\underset{\mathfrak{Z}_0}{\lesssim} S_1(n) +
S_2(n) + S_3(n)$, where
$$
S_1(n) = 2^{-\hat m_0(s_*+1/q-1/p_1)} \stackrel{(\ref{hat_mt})}{=}
n^{-s_*-1/q+1/p_1},
$$
$$
S_2(n) =2^{\mu_*k_*t(n)}\cdot 2^{-\hat m_{t(n)}(s_*+1/q-1/p_1)}
\stackrel{(\ref{hat_mt})}{\underset{\mathfrak{Z}_0}{\asymp}}
$$
$$
\asymp 2^{(\mu_*+\gamma_*(s_*+1/q-1/p_1))k_*t(n)}\cdot
n^{-s_*-1/q+1/p_1} \stackrel{(\ref{2tsss0}),
(\ref{eq3})}{\underset{\mathfrak{Z}_0}{\asymp}} n^{-\hat \theta
-(1/p_0-1/2)(1-\tilde \theta/s_*)},
$$
$$
S_3(n) = 2^{\mu_*k_*t_{***}(n)}\cdot
2^{-m^{(p_0')}_{t_{***}(n)}(s_*+1/q-1/p_1)}
\stackrel{(\ref{line_mt})}{\underset{\mathfrak{Z}_0}{\asymp}}
$$$$\asymp 2^{(\mu_*+ \gamma_*(s_*+1/q-1/p_1))k_*t_{***}(n)}\cdot
n^{-(s_*+1/q-1/p_1)p_0'/2}
\stackrel{(\ref{2tstar0})}{\underset{\mathfrak{Z}_0}{\asymp}}
n^{-\hat \theta p_0'/2}.
$$
\end{proof}

\section{Lower estimates for the linear widths of the classes $BX_{p_1}(\Omega)\cap BX_{p_0}(\Omega)$.}

As in \cite[\S 3]{vas_inters}, we suppose that there exist $c\ge
1$, $t_0\in \Z_+$ and functions $\varphi_j^{t,m}\in
X_{p_0}(\Omega) \cap X_{p_1}(\Omega)$ ($1\le j\le \nu_{t,m}$,
$t\ge t_0$, $m\in \Z_+$) with pairwise non-overlapping supports
such that
\begin{align}
\label{nu_tm} \nu_{t,m}=\lceil\nu'_{t,m}\rceil, \quad
\nu'_{t,m}=c^{-1}2^{\gamma_*k_*t}\cdot 2^m,
\end{align}
\begin{align}
\label{func_est} \begin{array}{c} \|\varphi_j^{t,m}\|
_{Y_q(\Omega)} = 1, \quad \|\varphi _j^{t,m}\| _{X_{p_0}(\Omega)}
\le c\cdot 2^{\alpha_*k_*t}\cdot2^{m\left(\frac 1q
-\frac{1}{p_0}\right)}, \\ \|\varphi _j^{t,m}\| _{X_{p_1}(\Omega)}
\le c\cdot 2^{-\mu_*k_*t}\cdot 2^{m\left(s_*+1/q-1/p_1\right)}.
\end{array}
\end{align}
We denote $\mathfrak{Z}_1 = (c, \, t_0, \, q, \, p_0, \, p_1,
\,k_*, \, s_*, \gamma_*, \, \alpha_*, \, \mu_*)$. The numbers
$\tilde \theta$ and $\hat \theta$ are defined by formula
(\ref{til_theta}), the numbers $j_0\in \N$ and $\theta_j\in \R$
($1\le j\le j_0$) are as in Definition \ref{theta_j} (we suppose
that one of its conditions holds).
\begin{Trm}
\label{low_est} Let (\ref{s1qp}), (\ref{mua}), (\ref{nu_tm}),
(\ref{func_est}) hold. Then
$$
\lambda_n(BX_{p_1}(\Omega)\cap BX_{p_0}(\Omega), \, Y_q(\Omega))
\underset{\mathfrak{Z}_1}{\gtrsim} n^{-\theta_j}, \quad 1\le j\le
j_0.
$$
\end{Trm}
\begin{proof}
The set $W_{t,m}$ is defined by (\ref{w_tm}).

As in \cite[p. 28]{vas_inters}, we get that
\begin{align}
\label{low_lin} \lambda_n(BX_{p_1}(\Omega)\cap BX_{p_0}(\Omega),
\, Y_q(\Omega)) \underset{\mathfrak{Z}_1}{\gtrsim}
\lambda_n(W_{t,m}, \, l_q^{\nu_{t,m}}).
\end{align}
We take $t=t_0$. Since $s_*>0$,
$s_*+\frac{1}{p_0}-\frac{1}{p_1}>0$, there is $\hat c=\hat
c(\mathfrak{Z}_1)$ such that for sufficiently large $m$ the
inclusion $\hat c\cdot 2^{-m(s_*+1/q-1/p_1)}\cdot
B_{p_1}^{\nu_{t_0,m}}\subset W_{t_0,m}$ holds. We take $m$ such
that $2n\le \nu'_{t_0,m} \underset{\mathfrak{Z}_1}{\lesssim} n$,
and for large $n$ we obtain the estimate
\begin{align}
\label{l1} \lambda_n(BX_{p_1}(\Omega)\cap BX_{p_0}(\Omega), \,
Y_q(\Omega))
\stackrel{(\ref{nu_tm}),(\ref{low_lin})}{\underset{\mathfrak{Z}_1}{\gtrsim}}
n^{-s_*-1/q+1/p_1} \lambda_n (B_{p_1}^{2n}, \, l_q^{2n}).
\end{align}
Let $p_1<2<q$. If $\frac 1q +\frac{1}{p_1}\ge 1$, we take $m$ such
that $2n^{q/2}\le \nu'_{t_0,m} \underset{\mathfrak{Z}_1}
{\lesssim} n^{q/2}$ and get
\begin{align}
\label{l2} \lambda_n(BX_{p_1}(\Omega)\cap BX_{p_0}(\Omega), \,
Y_q(\Omega)) \stackrel{(\ref{glusk_ln})}
{\underset{\mathfrak{Z}_1}{\gtrsim}} n^{-(s_*+1/q-1/p_1)q/2}.
\end{align}
If $\frac 1q +\frac{1}{p_1}\le 1$, we take $m$ such that
$2n^{p_1'/2}\le \nu'_{t_0,m} \underset{\mathfrak{Z}_1} {\lesssim}
n^{p_1'/2}$, and obtain
\begin{align}
\label{l3} \lambda_n(BX_{p_1}(\Omega)\cap BX_{p_0}(\Omega), \,
Y_q(\Omega)) \stackrel{(\ref{glusk_ln})}
{\underset{\mathfrak{Z}_1}{\gtrsim}} n^{-(s_*+1/q-1/p_1)p_1'/2}.
\end{align}

Now we apply (95) from \cite{vas_inters} and get
\begin{align}
\label{l4} \lambda_n(BX_{p_1}(\Omega)\cap BX_{p_0}(\Omega), \,
Y_q(\Omega)) \ge d_n(BX_{p_1}(\Omega)\cap BX_{p_0}(\Omega), \,
Y_q(\Omega)) \underset{\mathfrak{Z}_1}{\gtrsim} n^{-\tilde
\theta}.
\end{align}

Further we write $A\underset{\mathfrak{Z}_1}{\subset}B$ or
$B\underset{\mathfrak{Z}_1}{\supset}A$ if there exists
$\overline{c}(\mathfrak{Z}_1)\ge 1$ such that $A\subset
\overline{c}(\mathfrak{Z}_1)B$.

Let $m_t$ be defined by the equation
\begin{align}
\label{2akt} 2^{-\alpha_*k_*t}\cdot 2^{m_t(1/p_0-1/q)} =
2^{\mu_*k_*t}\cdot 2^{-m_t(s_*+1/q-1/p_1)}.
\end{align}

If $p_1\le p_0$, we have $2^{\mu_*k_*t}\cdot
2^{-m_t(s_*+1/q-1/p_1)}
B^{\nu_{t,[m_t]}}_{p_1}\underset{\mathfrak{Z}_1}{\subset}
2^{-\alpha_*k_*t}\cdot 2^{m_t(1/p_0-1/q)}
B^{\nu_{t,[m_t]}}_{p_0}$; hence,
\begin{align}
\label{ln} \lambda_n(W_{t,[m_t]}, \, l_q^{\nu_{t,[m_t]}})
\underset{\mathfrak{Z}_1}{\asymp} 2^{\mu_*k_*t}\cdot
2^{-m_t(s_*+1/q-1/p_1)} \lambda_n(B_{p_1}^{\nu_{t,[m_t]}}, \,
l_q^{\nu_{t,[m_t]}}).
\end{align}
If, in addition, $q\le 2$ or $p_1\ge 2$, we take $t$ such that
$2n\le \nu'_{t,[m_t]} \underset{\mathfrak{Z}_1}{\lesssim} n$;
applying (\ref{til_theta}), (\ref{glusk_ln1}),
(\ref{pietsch_stesin}), (\ref{nu_tm}), (\ref{low_lin}),
(\ref{2akt}), (\ref{ln}), we get the estimate
\begin{align}
\label{l5} \lambda_n(BX_{p_1}(\Omega)\cap BX_{p_0}(\Omega), \,
Y_q(\Omega)) \underset{\mathfrak{Z}_1}{\gtrsim} n^{-\hat \theta}.
\end{align}

If $p_1\ge p_0$, we have $2^{\mu_*k_*t}\cdot
2^{-m_t(s_*+1/q-1/p_1)}
B^{\nu_{t,[m_t]}}_{p_1}\underset{\mathfrak{Z}_1}{\supset}
2^{-\alpha_*k_*t}\cdot 2^{m_t(1/p_0-1/q)}
B^{\nu_{t,[m_t]}}_{p_0}$; hence,
\begin{align}
\label{ln1} \lambda_n(W_{t,[m_t]}, \, l_q^{\nu_{t,[m_t]}})
\underset{\mathfrak{Z}_1}{\asymp} 2^{-\alpha_*k_*t}\cdot
2^{m_t(1/p_0-1/q)} \lambda_n (B_{p_0}^{\nu_{t,[m_t]}}, \,
l_q^{\nu_{t,[m_t]}}).
\end{align}
If, in addition, $q\le 2$ or $p_0\ge 2$, we take $t$ such that
$2n\le \nu'_{t,[m_t]} \underset{\mathfrak{Z}_1}{\lesssim} n$;
applying (\ref{til_theta}), (\ref{glusk_ln1}),
(\ref{pietsch_stesin}), (\ref{nu_tm}), (\ref{low_lin}),
(\ref{2akt}), (\ref{ln1}), we obtain (\ref{l5}).

By (\ref{glusk_ln}), (\ref{glusk_ln1}), (\ref{pietsch_stesin}),
(\ref{l1}), (\ref{l4}), (\ref{l5}), we get the desired estimates
for cases 1--4 from Definition \ref{theta_j}.

\smallskip

{\bf The case $q>2$, $\frac 1q +\frac{1}{p_0}\ge 1$, $\frac 1q
+\frac{1}{p_1}\ge 1$.} For $p_1\le p_0$ we apply (\ref{ln}), for
$p_1\ge p_0$ we use (\ref{ln1}). In addition, we apply
(\ref{til_theta}), (\ref{glusk_ln}), (\ref{nu_tm}),
(\ref{low_lin}), (\ref{2akt}). We take $t$ such that $2n\le
\nu'_{t,[m_t]} \underset{\mathfrak{Z}_1}{\lesssim} n$ and get
\begin{align}
\label{00} \lambda_n(BX_{p_1}(\Omega)\cap BX_{p_0}(\Omega), \,
Y_q(\Omega)) \underset{\mathfrak{Z}_1}{\gtrsim} n^{-\hat
\theta-1/2+1/q};
\end{align}
taking $t$ such that $2n^{q/2}\le \nu'_{t,[m_t]}
\underset{\mathfrak{Z}_1}{\lesssim} n^{q/2}$, we get
\begin{align}
\label{000} \lambda_n(BX_{p_1}(\Omega)\cap BX_{p_0}(\Omega), \,
Y_q(\Omega)) \underset{\mathfrak{Z}_1}{\gtrsim} n^{-q\hat
\theta/2}.
\end{align}
This together with (\ref{glusk_ln}), (\ref{l1}), (\ref{l2}) gives
the lower estimate in case 5 of Definition \ref{theta_j}.

\smallskip

Before considering the other cases, we define $\tilde m_t$ by the
equation
\begin{align}
\label{2akt1} 2^{-\alpha_*k_*t}\cdot 2^{\tilde m_t(1/p_0-1/q)} =
2^{\mu_*k_*t}\cdot 2^{-\tilde m_t(s_*+1/q-1/p_1)}\cdot
2^{(1/p_0-1/p_1)\gamma_*k_*t}\cdot2^{\tilde m_t(1/p_0-1/p_1)}.
\end{align}

If $p_1\le p_0$, we have $2^{\mu_*k_*t}\cdot 2^{-\tilde
m_t(s_*+1/q-1/p_1)} B^{\nu_{t,[\tilde
m_t]}}_{p_1}\underset{\mathfrak{Z}_1}{\supset}
2^{-\alpha_*k_*t}\cdot 2^{\tilde m_t(1/p_0-1/q)} B^{\nu_{t,[\tilde
m_t]}}_{p_0}$; hence,
\begin{align}
\label{ln2} \lambda_n(W_{t,[\tilde m_t]}, \, l_q^{\nu_{t,[\tilde
m_t]}}) \underset{\mathfrak{Z}_1}{\asymp} 2^{-\alpha_*k_*t}\cdot
2^{\tilde m_t(1/p_0-1/q)} \lambda_n (B_{p_0}^{\nu_{t,[\tilde
m_t]}}, \, l_q^{\nu_{t,[\tilde m_t]}}).
\end{align}

If $p_1\ge p_0$, we have $2^{\mu_*k_*t}\cdot 2^{-\tilde
m_t(s_*+1/q-1/p_1)} B^{\nu_{t,[\tilde
m_t]}}_{p_1}\underset{\mathfrak{Z}_1}{\subset}
2^{-\alpha_*k_*t}\cdot 2^{\tilde m_t(1/p_0-1/q)} B^{\nu_{t,[\tilde
m_t]}}_{p_0}$; therefore,
\begin{align}
\label{ln3} \lambda_n(W_{t,[\tilde m_t]}, \, l_q^{\nu_{t,[\tilde
m_t]}}) \underset{\mathfrak{Z}_1}{\asymp} 2^{\mu_*k_*t}\cdot
2^{-\tilde m_t(s_*+1/q-1/p_1)} \lambda_n(B_{p_1}^{\nu_{t,[\tilde
m_t]}}, \, l_q^{\nu_{t,[\tilde m_t]}}).
\end{align}

{\bf The case $q>2$, $p_0<2<q$, $p_1\le 2<q$, $\frac 1q
+\frac{1}{p_0}\le 1$, $\frac 1q +\frac{1}{p_1}\le 1$, $p_0\le
p_1$.} We apply the estimate (\ref{ln3}), taking $2n\le
\nu'_{t,[\tilde m_t]} \underset{\mathfrak{Z}_1}{\lesssim} n$,
$2n^{p_1'/2}\le \nu'_{t,[\tilde m_t]}
\underset{\mathfrak{Z}_1}{\lesssim} n^{p_1'/2}$, and applying
(\ref{eq}). This together with (\ref{glusk_ln}), (\ref{nu_tm}),
(\ref{low_lin}), (\ref{2akt1}) yields
\begin{align}
\label{11} \lambda_n(BX_{p_1}(\Omega)\cap BX_{p_0}(\Omega), \,
Y_q(\Omega)) \underset{\mathfrak{Z}_1}{\gtrsim} n^{-\hat
\theta-(1/p_0-1/p_1)(1-\tilde \theta/s_*)-1/p_1+1/2},
\end{align}
\begin{align}
\label{22} \lambda_n(BX_{p_1}(\Omega)\cap BX_{p_0}(\Omega), \,
Y_q(\Omega)) \underset{\mathfrak{Z}_1}{\gtrsim} n^{-(\hat \theta+
(1/p_0-1/p_1)(1-\tilde \theta/s_*))p_1'/2}.
\end{align}
Now we use the estimate (\ref{ln1}), taking $2n^{p_0'/2}\le
\nu'_{t,[m_t]} \underset{\mathfrak{Z}_1}{\lesssim} n^{p_0'/2}$.
This together with (\ref{glusk_ln}), (\ref{nu_tm}),
(\ref{low_lin}), (\ref{2akt}) implies
\begin{align}
\label{55} \lambda_n(BX_{p_1}(\Omega)\cap BX_{p_0}(\Omega), \,
Y_q(\Omega)) \underset{\mathfrak{Z}_1}{\gtrsim} n^{-\hat\theta
p_0'/2}.
\end{align}
From (\ref{l1}), (\ref{l3}), (\ref{11}), (\ref{22}), (\ref{55}) we
obtain the lower estimate in cases 6 and 8 of Definition
\ref{theta_j}.

{\bf The case $q>2$, $p_0\le 2<q$, $p_1<2<q$, $\frac 1q
+\frac{1}{p_0}\le 1$, $\frac 1q +\frac{1}{p_1}\le 1$, $p_0\ge
p_1$.} We apply (\ref{ln2}), taking $2n\le \nu'_{t,[\tilde m_t]}
\underset{\mathfrak{Z}_1}{\lesssim} n$, $2n^{p_0'/2}\le
\nu'_{t,[\tilde m_t]} \underset{\mathfrak{Z}_1}{\lesssim}
n^{p_0'/2}$, and applying (\ref{eq1}). This together with
(\ref{glusk_ln}), (\ref{nu_tm}), (\ref{low_lin}), (\ref{2akt1})
yields
\begin{align}
\label{33} \lambda_n(BX_{p_1}(\Omega)\cap BX_{p_0}(\Omega), \,
Y_q(\Omega)) \underset{\mathfrak{Z}_1}{\gtrsim} n^{-\hat
\theta-(1/p_1-1/p_0)\tilde \theta/s_*-1/p_0+1/2},
\end{align}
\begin{align}
\label{44} \lambda_n(BX_{p_1}(\Omega)\cap BX_{p_0}(\Omega), \,
Y_q(\Omega)) \underset{\mathfrak{Z}_1}{\gtrsim} n^{-(\hat \theta+
(1/p_1-1/p_0)\tilde \theta/s_*)p_0'/2}.
\end{align}
Now we apply (\ref{ln}), taking $2n^{p_1'/2}\le \nu'_{t,[m_t]}
\underset{\mathfrak{Z}_1}{\lesssim} n^{p_1'/2}$. We use
(\ref{glusk_ln}), (\ref{nu_tm}), (\ref{low_lin}), (\ref{2akt}) and
get
\begin{align}
\label{66} \lambda_n(BX_{p_1}(\Omega)\cap BX_{p_0}(\Omega), \,
Y_q(\Omega)) \underset{\mathfrak{Z}_1}{\gtrsim} n^{-\hat\theta
p_1'/2}.
\end{align}
From (\ref{l1}), (\ref{l3}), (\ref{33}), (\ref{44}), (\ref{66}) we
obtain the lower estimate in cases 7 and 8 of Definition
\ref{theta_j}.

{\bf The case $q>2$, $p_0<2<q$, $p_1\le 2<q$, $\frac 1q
+\frac{1}{p_0}> 1$, $\frac 1q +\frac{1}{p_1}< 1$.} Then $p_0<
p_1$. We use (\ref{ln3}), taking $2n\le \nu'_{t,[\tilde m_t]}
\underset{\mathfrak{Z}_1}{\lesssim} n$, $2n^{p_1'/2}\le
\nu'_{t,[\tilde m_t]} \underset{\mathfrak{Z}_1}{\lesssim}
n^{p_1'/2}$, and applying (\ref{eq}). Hence, we get (\ref{11}),
(\ref{22}). Further, we use (\ref{ln1}) with $2n\le \nu'_{t,[m_t]}
\underset{\mathfrak{Z}_1}{\lesssim} n$ and $2n^{q/2}\le
\nu'_{t,[m_t]} \underset{\mathfrak{Z}_1}{\lesssim} n^{q/2}$. We
obtain (\ref{00}), (\ref{000}). This together with (\ref{l1}),
(\ref{l3}) yields the estimate in case 9 of Definition
\ref{theta_j}.

{\bf The case $q>2$, $p_0\le 2<q$, $p_1<2<q$, $\frac 1q
+\frac{1}{p_0}< 1$, $\frac 1q +\frac{1}{p_1}> 1$.} Then $p_0>
p_1$. We use (\ref{ln2}) with $2n\le \nu'_{t,[\tilde m_t]}
\underset{\mathfrak{Z}_1}{\lesssim} n$, $2n^{p_0'/2}\le
\nu'_{t,[\tilde m_t]} \underset{\mathfrak{Z}_1}{\lesssim}
n^{p_0'/2}$, and take into account (\ref{eq1}). Hence, we get
(\ref{33}), (\ref{44}). Further, we use (\ref{ln}) with $2n\le
\nu'_{t,[m_t]} \underset{\mathfrak{Z}_1}{\lesssim} n$ and
$2n^{q/2}\le \nu'_{t,[m_t]} \underset{\mathfrak{Z}_1}{\lesssim}
n^{q/2}$. Now, we get (\ref{00}), (\ref{000}). This together with
(\ref{l1}), (\ref{l2}) gives the lower estimate for case 10 of
Definition \ref{theta_j}.

{\bf The case $p_1<2<p_0<q$, $\frac 1q+\frac{1}{p_1}\le 1$.} We
have $p_1<p_0$. Applying (\ref{ln}) with $2n^{p_1'/2}\le
\nu'_{t,[m_t]}\underset{\mathfrak{Z}_1}{\lesssim} n^{p_1'/2}$, we
get (\ref{66}).

Let $\overline{m}_t$ be defined by the equation
\begin{align}
\label{ovr_mt} 2^{-\alpha_*k_*t}\cdot
2^{-\overline{m}_t(1/q-1/p_0)} = 2^{\mu_*k_*t}\cdot
2^{-\overline{m}_t(s_*+1/q-1/p_1)}n^{-1/2}\cdot
2^{\gamma_*k_*t/p_1'}\cdot 2^{\overline{m}_t/p_1'}.
\end{align}
We take $t(n)$ such that
\begin{align}
\label{nu_n} 2n\le \nu'_{t(n),[\overline{m}_{t(n)}]}
\underset{\mathfrak{Z}_1}{\lesssim} n.
\end{align}
Let
\begin{align}
\label{kin} \begin{array}{c} k_0(n)=2^{-\alpha_*k_*t(n)}\cdot
2^{-\overline{m}_{t(n)}(1/q-1/p_0)}, \\ k_1(n)
=2^{\mu_*k_*t(n)}\cdot 2^{-\overline{m}_{t(n)}(s_*+1/q-1/p_1)}, \;
\nu(n)=[\nu'_{t(n),[\overline{m}_{t(n)}]}].\end{array}
\end{align}
By (\ref{eq2}), (\ref{nu_tm}), (\ref{ovr_mt}), (\ref{nu_n}), we
get
\begin{align}
\label{k0n} k_0(n) \underset{\mathfrak{Z}_1}{\asymp} n^{-\hat
\theta-(1/p_1-1/2)\tilde \theta/s_*}.
\end{align}
We prove the estimate
\begin{align}
\label{mm} \lambda_n(k_1(n)B_{p_1}^{\nu(n)}\cap
k_0(n)B_{p_0}^{\nu(n)}, \, l_q^{\nu(n)})
\underset{\mathfrak{Z}_1}{\gtrsim} k_0(n).
\end{align}
This together with (\ref{low_lin}), (\ref{k0n}) implies that
\begin{align}
\label{ll} \lambda_n(BX_{p_1}(\Omega)\cap BX_{p_0}(\Omega), \,
Y_q(\Omega)) \underset{\mathfrak{Z}_1}{\gtrsim} n^{-\hat
\theta-(1/p_1-1/2)\tilde \theta/s_*}.
\end{align}

In order to prove (\ref{mm}), it is sufficient to check the
inclusion
$$
 k_0(n)B_2^{\nu(n)}
\underset{\mathfrak{Z}_1}{\subset} k_1(n)B_{p_1}^{\nu(n)}\cap
k_0(n)B_{p_0}^{\nu(n)}
$$
and apply (\ref{glusk_ln1}). Since $p_0>2$, $k_0(n)B_2^{\nu(n)}
\subset k_0(n)B_{p_0}^{\nu(n)}$. We show that $k_0(n)B_2^{\nu(n)}
\underset{\mathfrak{Z}_1}{\subset} k_1(n)B_{p_1}^{\nu(n)}$. To
this end, it suffices to check that
$k_1(n)\underset{\mathfrak{Z}_1}{\gtrsim}
k_0(n)\nu(n)^{1/p_1-1/2}$; this holds by (\ref{nu_tm}),
(\ref{ovr_mt}), (\ref{nu_n}) and (\ref{kin}). This completes the
proof of (\ref{ll}).

From (\ref{l1}), (\ref{l3}), (\ref{66}) and (\ref{ll}) we get the
lower estimates in case 11 of Definition \ref{theta_j}.

{\bf The case $p_0<2<p_1<q$, $\frac 1q+\frac{1}{p_0}\le 1$.} We
have $p_1>p_0$. Applying (\ref{ln1}) with $2n^{p_0'/2}\le
\nu'_{t,[m_t]}\underset{\mathfrak{Z}_1}{\lesssim} n^{p_0'/2}$, we
get (\ref{55}).

Let $\overline{m}_t$ be defined by the equation
\begin{align}
\label{ovr_mt1} 2^{-\alpha_*k_*t}\cdot
2^{-\overline{m}_t(1/q-1/p_0)}n^{-1/2}\cdot
2^{\gamma_*k_*t/p_0'}\cdot 2^{\overline{m}_t/p_0'} =
2^{\mu_*k_*t}\cdot 2^{-\overline{m}_t(s_*+1/q-1/p_1)};
\end{align}
$t(n)$ is such as in (\ref{nu_n}). We use notation (\ref{kin}).

From (\ref{eq3}), (\ref{nu_tm}), (\ref{nu_n}) and (\ref{ovr_mt1})
it follows that
\begin{align}
\label{k1n} k_1(n) \underset{\mathfrak{Z}_1}{\asymp} n^{-\hat
\theta-(1/p_0-1/2)(1-\tilde \theta/s_*)}.
\end{align}

We prove the estimate
\begin{align}
\label{lnk1} \lambda_n(k_1(n)B_{p_1}^{\nu(n)}\cap
k_0(n)B_{p_0}^{\nu(n)}, \, l_q^{\nu(n)})
\underset{\mathfrak{Z}_1}{\gtrsim} k_1(n).
\end{align}
This together with (\ref{low_lin}), (\ref{k1n}) implies that
\begin{align}
\label{ll1} \lambda_n(BX_{p_1}(\Omega)\cap BX_{p_0}(\Omega), \,
Y_q(\Omega)) \underset{\mathfrak{Z}_1}{\gtrsim} n^{-\hat
\theta-(1/p_0-1/2)(1-\tilde \theta/s_*)}.
\end{align}

In order to prove (\ref{lnk1}), it is sufficient to check the
inclusion
$$
 k_1(n)B_2^{\nu(n)} \underset{\mathfrak{Z}_1}{\subset}
k_1(n)B_{p_1}^{\nu(n)}\cap k_0(n)B_{p_0}^{\nu(n)}
$$
and apply (\ref{glusk_ln1}). Since $p_1>2$, we have
$k_1(n)B_2^{\nu(n)} \subset k_1(n)B_{p_1}^{\nu(n)}$. We show that
$k_1(n)B_2^{\nu(n)} \underset{\mathfrak{Z}_1}{\subset}
k_0(n)B_{p_0}^{\nu(n)}$. To this end, it is sufficient to check
that $k_0(n)\underset{\mathfrak{Z}_1}{\gtrsim}
k_1(n)\nu(n)^{1/p_0-1/2}$; it holds by (\ref{nu_tm}),
(\ref{nu_n}), (\ref{kin}), (\ref{ovr_mt1}). This completes the
proof of (\ref{ll1}).

From (\ref{l1}), (\ref{55}) and (\ref{ll1}) we get the lower
estimates for case 12 of Definition \ref{theta_j}.
\end{proof}


\begin{Biblio}
\bibitem{vas_inters} A.A. Vasil'eva, ``Kolmogorov widths of weighted Sobolev classes on a multi-dimensional domain with conditions on the derivatives of
order $r$ and zero'', arXiv:2004.06013v2.

\bibitem{trieb_mat_sb} H. Triebel, ``Interpolation properties of $\varepsilon$-entropy and diameters.
Geometric characteristics of imbedding for function spaces of
Sobolev–Besov type'', {\it Math. USSR-Sb.}, {\bf 27}:1 (1975),
23--37.

\bibitem{triebel} H. Triebel, {\it Interpolation theory. Function spaces. Differential
operators}. Mir, Moscow, 1980.

\bibitem{lo1} P.I. Lizorkin, M. Otelbaev, ``Imbedding theorems and compactness
for spaces of Sobolev type with weights'', {\it Math. USSR-Sb.}
{\bf 36}:3 (1980), 331--349.

\bibitem{lo2} P.I. Lizorkin, M. Otelbaev, ``Imbedding theorems and
compactness for spaces of Sobolev type with weights. II'', {\it
Math. USSR-Sb.} {\bf 40}:1 (1981), 51--77.

\bibitem{lo3} P.I. Lizorkin, M. Otelbaev, ``Estimates of approximate numbers
of the imbedding operators for spaces of Sobolev type with
weights'', {\it Proc. Steklov Inst. Math.}, {\bf 170} (1987),
245--266.

\bibitem{myn_otel} K.~Mynbaev, M.~Otelbaev, {\it Weighted function spaces and the
spectrum of differential operators.} Nauka, Moscow, 1988.

\bibitem{boy_1} I.V.~Boykov, ``Approximation of some classes
of functions by local splines'', {\it Comput. Math. Math. Phys.}
{\bf 38}:1 (1998), 21--29.

\bibitem{ait_kus1} M.S.~Aitenova, L.K.~Kusainova, ``On the asymptotics of the distribution of approximation
numbers of embeddings of weighted Sobolev classes. I'', {\it Mat.
Zh.} {\bf 2}:1 (2002), 3--9.
\bibitem{ait_kus2} M.S.~Aitenova, L.K.~Kusainova, ``On the asymptotics of the distribution of approximation
numbers of embeddings of weighted Sobolev classes. II'', {\it Mat.
Zh.} {\bf 2}:2 (2002), 7--14.

\bibitem{triebel12} H. Triebel, ``Entropy and approximation numbers of limiting embeddings, an approach via Hardy
inequalities and quadratic forms''. {\it J. Approx. Theory}. {\bf
164} (2012), no. 1, 31--46.

\bibitem{mieth1} T. Mieth, ``Entropy and approximation numbers of embeddings of
weighted Sobolev spaces''. {\it J. Appr. Theory}. {\bf 192}
(2015), 250--272.

\bibitem{mieth2} T. Mieth, ``Entropy and approximation numbers of weighted
Sobolev spaces via bracketing''. {\it J. Funct. Anal.} {\bf 270}
(2016), 4322--4339.

\bibitem{vas_width_raspr} A.A. Vasil'eva, ``Widths of function classes on sets with tree-like
structure'', {\it J. Appr. Theory}, {\bf 192} (2015), 19--59.



\bibitem{bib_gluskin} E.D. Gluskin, ``Norms of random matrices and diameters
of finite-dimensional sets'', {\it Math. USSR-Sb.}, {\bf 48}:1
(1984), 173--182.

\bibitem{pietsch1} A. Pietsch, ``$s$-numbers of operators in Banach space'', {\it Studia Math.},
{\bf 51} (1974), 201--223.

\bibitem{stesin} M.I. Stesin, ``Aleksandrov diameters of finite-dimensional sets
and of classes of smooth functions'', {\it Dokl. Akad. Nauk SSSR},
{\bf 220}:6 (1975), 1278--1281 [Soviet Math. Dokl.].

\bibitem{galeev1} E.M.~Galeev, ``The Kolmogorov diameter of the intersection of classes of periodic
functions and of finite-dimensional sets'', {\it Math. Notes},
{\bf 29}:5 (1981), 382--388.

\end{Biblio}
\end{document}